\documentclass[11pt,leqno]{amsart}

\usepackage[english]{babel}
\usepackage{url}
\usepackage[latin1]{inputenc}

\usepackage[all,2cell,emtex]{xy}
\usepackage{amssymb,amsmath,bbm,xr,a4wide,color,stmaryrd}
\usepackage[mathscr]{euscript}
\usepackage{mathrsfs}

\newcommand{\sm}{\mathrm{Sm}_k}
\newcommand{\smc}{\widetilde{\mathrm{Cor}}_k}
\newcommand{\smcR}[1]{\widetilde{\mathrm{Cor}}_{k,{#1}}}
\newcommand{\smcV}{\mathrm{Cor}_k}
\newcommand{\smcVR}[1]{\mathrm{Cor}_{k,{#1}}}

\newcommand{\psh}{\mathrm{PSh}(k,R)}
\newcommand{\psht}{\widetilde{\mathrm{PSh}}(k,R)}
\newcommand{\pshtV}{\mathrm{PSh}^{\mathrm{tr}}(k,R)}
\newcommand{\pshfr}{\mathrm{PSh}^{\mathrm{Fr}}(k,R)}

\newcommand{\sh}{\mathrm{Sh}_t(k,R)}
\newcommand{\sht}{\widetilde{\mathrm{Sh}}_t(k,R)}
\newcommand{\shtV}{\mathrm{Sh}^{\mathrm{tr}}_t(k,R)}

\newcommand{\shx}[1]{\mathrm{Sh}_{#1}(k,R)}
\newcommand{\shtx}[1]{\widetilde{\mathrm{Sh}}_{#1}(k,R)}
\newcommand{\shtVx}[1]{\mathrm{Sh}^{\mathrm{tr}}_{#1}(k,R)}

\DeclareMathOperator{\corr}{\tilde {\mathrm{c}}}
\DeclareMathOperator{\Icorr}{\mathrm{I}\tilde{\mathrm{c}}}
\DeclareMathOperator{\corV}{\mathrm{c}}

\newcommand{\DAe}{\Der_{\AA^1}^{\mathrm{eff}}(k,R)}
\newcommand{\DAeZ}{\Der_{\AA^1}^{\mathrm{eff}}(k)}
\newcommand{\DMte}{\widetilde{\mathrm{DM}}{}^{\mathrm{eff}}\!(k,R)}
\newcommand{\DMteZ}{\widetilde{\mathrm{DM}}{}^{\mathrm{eff}}\!(k)}
\newcommand{\DMe}{\mathrm{DM}^{\mathrm{eff}}(k,R)}

\newcommand{\DAex}[1]{\Der_{\AA^1,#1}^{\mathrm{eff}}(k,R)}
\newcommand{\DMtex}[1]{\widetilde{\mathrm{DM}}{}_{#1}^{\mathrm{eff}}(k,R)}
\newcommand{\DMex}[1]{\mathrm{DM}_{#1}^{\mathrm{eff}}(k,R)}

\newcommand{\DMtegm}{\widetilde{\mathrm{DM}}^{\mathrm{eff}}_{gm}(k,\ZZ)}
\newcommand{\DMteZZ}{\widetilde{\mathrm{DM}}(k,\ZZ)}

\newcommand{\shtZZ}{\widetilde{\mathrm{Sh}}(k,\ZZ)}

\newcommand{\DA}{\Der_{\AA^1}(k,R)}
\newcommand{\DMt}{\widetilde{\mathrm{DM}}(k,R)}
\newcommand{\DM}{\mathrm{DM}(k,R)}

\newcommand{\DAx}[1]{\Der_{\AA^1,#1}(k,R)}
\newcommand{\DMtx}[1]{\widetilde{\mathrm{DM}}_{#1}(k,R)}
\newcommand{\DMx}[1]{\mathrm{DM}_{#1}(k,R)}

\newcommand{\spt}{\widetilde{\mathrm{Sp}}_t(k,R)}

\def\KM{\mathrm{K}^{\mathrm{M}}} 
\def\KMW{\mathrm{K}^{\mathrm{M\hspace{-.2ex}W}}} 

\def\sKMW{\mathbf{K}^{\mathrm{M\hspace{-.2ex}W}}} 
\def\Wi{\mathbf{W}}

\def\H{\mathrm{H}}
\DeclareMathOperator{\Hom}{Hom}

\DeclareMathOperator{\uHom}{\underline{Hom}} 

\DeclareMathOperator{\wCH}{\widetilde{CH}}
\DeclareMathOperator{\spec}{Spec}

\DeclareMathOperator{\ilim}{\varinjlim}

\newcommand{\tO}{\tilde{\mathcal O}} 

\newcommand{\prep}[1]{{\corr}(#1)}
\newcommand{\prepNP}{{\corr}}
\newcommand{\rep}[1]{\tilde \ZZ_t(#1)}

\newcommand{\repR}[1]{\tilde {R}_t(#1)}
\newcommand{\IrepR}[1]{\mathrm{I}\tilde{R}_t(#1)}

\DeclareMathOperator{\otr}{\tilde \otimes}

\newcommand{\mot}{\tilde {\mathrm{M}}}
\newcommand{\motV}{\mathrm{M}}

\newcommand{\sus}[1]{\mathrm{C}_*^{sing}(#1)}

\DeclareMathOperator{\Th}{\mathrm{Th}}  

\DeclareMathOperator{\Comp}{C}
\DeclareMathOperator{\K}{K}
\DeclareMathOperator{\Der}{D}

\newcommand{\derL}{\mathbf{L}}
\newcommand{\derR}{\mathbf{R}}

\newcommand{\cX}{\mathcal X}

\newcommand{\NN} {\mathbb N}
\newcommand{\ZZ} {\mathbb Z}

\newcommand{\OO}{\mathcal O}
\renewcommand{\AA} {\mathbb A}
\newcommand{\PP} {\mathbb P}
\newcommand{\GG} {\mathbb{G}_m}

\newcommand{\GGx}[1] {\mathbb{G}_{m,#1}}

\newcommand{\nis}{\mathrm{Nis}}
\newcommand{\zar}{\mathrm{Zar}}
\newcommand{\et}{\mathrm{\acute et}}

\title{$\mathrm{MW}$-motivic complexes}

\author{Fr\'ed\'eric D\'eglise}
\address{Institut Math\'ematique de Bourgogne - UMR 5584, Universit\'e de Bourgogne, 9 avenue Alain Savary, BP 47870, 21078 Dijon Cedex, France}
\email{frederic.deglise@ens-lyon.fr}
\urladdr{http://perso.ens-lyon.fr/frederic.deglise/}
 
\author{Jean Fasel}
\address{Institut Fourier - UMR 5582, Universit\'e Grenoble-Alpes, CS 40700, 38058 Grenoble Cedex 9, France}
\email{jean.fasel@gmail.com}
\urladdr{https://www.uni-due.de/~adc301m/staff.uni-duisburg-essen.de/Home.html}

\thanks{The first author was supported by the ANR (grant No. ANR-12-BS01-0002).}
\date{\today}

\newtheorem{thm}{Theorem}[subsection]
\newtheorem{prop}[thm]{Proposition}
\newtheorem{lm}[thm]{Lemma}
\newtheorem{cor}[thm]{Corollary}

\theoremstyle{remark} 
\newtheorem{rem}[thm]{Remark}

\newtheorem{ex}[thm]{Example}

\theoremstyle{definition} 
\newtheorem{df}[thm]{Definition}
\newtheorem{num}[thm]{}

\numberwithin{equation}{thm}

\newtheorem{thm*}{Theorem}

\begin{document}

\begin{abstract}
The aim of this work is to develop a theory parallel
 to that of motivic complexes based on cycles
 and correspondences with coefficients in quadratic forms.
 This framework is closer to the point of view of $\AA^1$-homotopy than the original
 one envisioned by Beilinson and set up by Voevodsky.
\end{abstract}

\maketitle

\setcounter{tocdepth}{3}
\tableofcontents

\section*{Introduction}

The aim of this paper is to define the various categories of $\mathrm{MW}$-motives built out of the category of finite Chow-Witt correspondences constructed in \cite{Calmes14b}, and to study the motivic cohomology groups intrinsic to these categories. In Section \ref{sec:MWtransfers}, we start with a quick remainder of the basic properties of the category $\smc$. We then proceed with our first important result, namely that the sheaf (in either the Nisnevich or the \'etale topologies) associated to a $\mathrm{MW}$-presheaf, i.e. an additive functor $\smc\to \mathrm{Ab}$,  is a $\mathrm{MW}$-sheaf. The method follows closely Voevodsky's method and relies on Lemma \ref{lm:corr_main}. We also discuss the monoidal structure on the category of $\mathrm{MW}$-sheaves. In the second part of the paper, we prove the analogue for $\mathrm{MW}$-sheaves of a famous theorem of Voevodsky saying that the sheaf (with transfers) associated to a homotopy invariant presheaf with transfers is strictly $\AA^1$-invariant. Our method here is quite lazy. We heavily rely on the fact that an analogue theorem holds for quasi-stable sheaves with framed transfers by \cite[Theorem 1.1]{Garkusha15}. Having this theorem at hand, it suffices to construct a functor from the category of linear framed presheaves to $\smc$ to prove the theorem. This functor is of independent interest and this is the reason why we take this shortcut. There is currently work by Hakon Andreas Kolderup to give an independant (but still relying on ideas of Panin-Garkusha) proof of this theorem. In Section \ref{sec:MWmotives}, we finally pass to the construction of the categories of $\mathrm{MW}$-motives starting with a study of different model structures on the category of possibly unbounded complexes of $\mathrm{MW}$-sheaves. The ideas here are closely related to \cite{CD3}. The category of effective motives $\DMte$  (with coefficients in a ring $R$) is then defined as the category of $\AA^1$-local objects in this category of complexes. Using the analogue of Voevodsky's theorem proved in Section \ref{sec:framed}, these objects are easily characterized by the fact that their homology sheaves are strictly $\AA^1$-invariant. This allows as usual to give an explicit $\AA^1$-localization functor, defined in terms of the Suslin (total) complex. The category of geometric objects is as in the classical case the subcategory of compact objects of $\DMte$. Our next step is the formal inversion of the Tate motive in $\DMte$ to obtain the stable category of $\mathrm{MW}$-motives $\DMt$ (with coefficients in $R$). We can then consider motivic cohomology as groups of extensions in this category, a point of view which allows to prove in Section \ref{sec:MWcohom} many basic property of this version of motivic cohomology, including a commutativity statement and a comparison theorem between motivic cohomology and Chow-Witt groups.

\section*{Acknowledgments}

The authors would like to warmly thank G. Garkusha for remarks on previous versions of this article and for interesting conversations about framed correspondences. We also would like to thank T. Bachmann for numerous interesting conversations and useful remarks.

\section*{Conventions} \label{conventions}

In all this work,
 we will fix a base field $k$ assumed to be infinite perfect.
 All schemes considered will be assumed
 to be separated of finite type over $k$,
 unless explicitly stated.

We will fix a ring of coefficients $R$.
 We will also consider a Grothendieck topology $t$ on the site of
 smooth $k$-schemes, which in practice will be either the Nisnevich
 of the \'etale topology. In section 3 and 4 we will restrict
 to these two latter cases.

\section{$\mathrm{MW}$-transfers on sheaves}\label{sec:MWtransfers}

\subsection{Reminder on $\mathrm{MW}$-correspondences}

\begin{num}
We will use the definitions and constructions of \cite{Calmes14b}.

In particular, for any smooth schemes $X$ and $Y$ (with $Y$ connected of dimension $d$),
 we consider the following group of \emph{finite $\mathrm{MW}$-correspondences} 
 from $X$ to $Y$:
\begin{equation}\label{eq:def_corr}
\corr(X,Y):=\ilim_T \wCH^d_T\!\big(X \times Y,\omega_{Y}\big)
\end{equation}
where $T$ runs over the ordered set of reduced (but not necessarily irreducible) closed subschemes in
 $X \times Y$ whose projection to $X$ is finite equidimensional and $\omega_Y$ is the pull back of the canonical sheaf of $Y$ along the projection to the second factor.
 This definition is extended to the case where $Y$ is non connected
 by additivity. When considering the coefficients ring $R$, we put:
$$
\corr(X,Y)_R:=\corr(X,Y) \otimes_\ZZ R.
$$
 In the sequel, we drop the index $R$ from the notation when there is no possible confusion. 
 
Because there is a natural morphism from Chow-Witt groups
 (twisted by any line bundle) to Chow groups, we get a canonical map:
\begin{equation}\label{eq:from_corr2corV}
\pi_{XY}:\corr(X,Y) \rightarrow \corV(X,Y)
\end{equation}
for any smooth schemes $X$ and $Y$, where the right hand side
 is the group of Voevodsky's finite correspondences which
 is compatible to the composition --- see \emph{loc. cit.}
 Remark 4.12. Let us recall the following result. 
\end{num}

\begin{lm}\label{lm:corr&corrV_upto2tor}
If $2\in R^\times$, the induced map
$$
\pi_{XY}:\corr(X,Y) \rightarrow \corV(X,Y)
$$
is a split epimorphism.
\end{lm}
The lemma comes from the basic fact that the following composite map
$$
\KM_n(F) \xrightarrow{(1)} \KMW_n(F,\mathcal L)
 \xrightarrow{(2)} \KM_n(F)
$$
is multiplication by $2$, where (1) is the map from Milnor K-theory
 of a field $F$ to Milnor-Witt K-theory of $F$ twisted by the $1$-dimensional
 $F$-vector space $\mathcal L$ described in \cite[\S 1]{Calmes14b} and (2) is the map killing $\eta$
 (see the discussion in \emph{loc. cit.} after Definition 3.1).

\begin{rem}\label{rem:finite_corr&plim}
\begin{enumerate}
\item In fact, a finite $\mathrm{MW}$-correspondence amounts to a finite correspondence
 $\alpha$ together with a quadratic form over the function field
 of each irreducible component of the support of $\alpha$
 satisfying some condition related with residues;
 see \cite[Def. 4.6]{Calmes14b}.
\item Every finite $\mathrm{MW}$-correspondence between smooth schemes $X$ and $Y$ has a well defined support (\cite[Definition 4.6]{Calmes14b}). Roughly speaking, it is the minimal closed subset of $X\times Y$ on which the correspondence is defined.
\item Recall that the Chow-Witt group in degree $n$ 
 of a smooth $k$-scheme $X$ can be defined as the $n$-th Nisnevich cohomology
 group of the $n$-th unramified Milnor-Witt sheaf $\sKMW_n$ (this cohomology being computed using an explicit flabby resolution of $\sKMW_n$).
 This implies that the definition can be uniquely extended to the case
 where $X$ is an essentially smooth $k$-scheme.
 Accordingly, one can extend the definition of finite 
 $\mathrm{MW}$-correspondences to the case of essentially smooth
 $k$-schemes using formula \eqref{eq:def_corr}. The definition
 of composition obviously extends to that generalized setting.
 We will use that fact in the proof of Lemma \ref{lm:corr_main}.
\item Consider the notations of the previous point.
 Assume that the essentially smooth $k$-scheme
 $X$ is the projective limit of a projective system of essentially
 smooth $k$-schemes $(X_i)_{i \in I}$. Then the canonical map:
\[
\Big({\ilim}_{i \in I^{op}} \corr(X_i,Y)\Big) \longrightarrow \corr(X,Y)
\]
is an isomorphism. This readily follows from formula \eqref{eq:def_corr}
 and the fact that Chow-Witt groups, as Nisnevich cohomology, commute
 with projective limits of schemes. See also \cite[\S 5.1]{Calmes14b} for an extended discussion of these facts. 
 \item For any smooth schemes $X$ and $Y$, the group $\corr(X,Y)$ is endowed with a structure of a left $\sKMW_0(X)$-module and a right $\sKMW_0(Y)$-module (\cite[Example 4.10]{Calmes14b}). Pulling back along $X\to \spec k$, it follows that $\corr(X,Y)$ is a left $\sKMW_0(k)$-module and it is readily verified that the category $\smc $ is in fact $\sKMW_0(k)$-linear. Consequently, we can also consider $\sKMW_0(k)$-algebras as coefficient rings. 
\end{enumerate}
\end{rem}

\begin{num}
Recall from \emph{loc. cit.} that there is a composition product
 of $\mathrm{MW}$-correspondences which is compatible with the projection
 map $\pi_{XY}$.
\end{num}
\begin{df}
We denote by $\smc$ (resp. $\smcV$) the additive category whose objects
 are smooth schemes and morphisms are finite $\mathrm{MW}$-correspondences
 (resp. correspondences). If $R$ is a ring,
 we let $\smcR R$ (resp. $\smcVR R$) be
 the category $\smc \otimes_{\ZZ} R$ (resp. $\smcV \otimes_{\ZZ} R$).

We denote by
\begin{equation}\label{eq:from_smc2smcV}
\pi:\smc \rightarrow \smcV
\end{equation}
the additive functor which is the identity on objects and the map
 $\pi_{XY}$ on morphisms.
\end{df}

As a corollary of the above lemma, the induced functor
$$
\pi:\smcR R \rightarrow \smcVR R,
$$
is full when $2\in R^\times$. Note that the corresponding result without inverting $2$ is wrong by \cite[Remark 4.15]{Calmes14b}.

\begin{num}
The external product of finite $\mathrm{MW}$-correspondences induces
 a symmetric monoidal structure on $\smc$ which on objects is given by
 the cartesian product of $k$-schemes. One can check that the
 functor $\pi$ is symmetric monoidal, for the usual
 symmetric monoidal structure on the category $\smcV$.
 
Finally, the graph of any morphism $f:X \rightarrow Y$ can be
 seen not only as a finite correspondence $\gamma(f)$ from $X$ to $Y$
 but also as a finite $\mathrm{MW}$-correspondence $\tilde \gamma(f)$
 such that $\pi \tilde \gamma(f)=\gamma(f)$. One obtains in this
 way a canonical functor:
\begin{equation}\label{eq:from_sm2smc}
\tilde \gamma:\sm \rightarrow \smc
\end{equation}
which is faithful, symmetric monoidal, and such that
 $\pi \circ \tilde \gamma=\gamma$. 
\end{num}

\subsection{$\mathrm{MW}$-transfers}

\begin{df}
We let $\psht$ (resp. $\pshtV$, resp. $\psh$) be the category of additive presheaves
 of $R$-modules on $\smc$ (resp. $\smcV$, resp. $\sm$). Objects of
 $\psht$ will be simply called \emph{$\mathrm{MW}$-presheaves}.
\end{df}

\begin{df}\label{def:representablepsht}
We denote by $\prepNP_R(X)$ the representable presheaf $Y\mapsto \corr (Y,X)\otimes_{\ZZ} R$. As usual, we also write $\prep X$ in place of  $\prepNP_R(X)$ in case the context is clear.
\end{df}

The category of $\mathrm{MW}$-presheaves is
 an abelian Grothendieck category.\footnote{\label{fn:Grothendieck}Recall
 that an abelian category
 is called Grothendieck abelian if it admits a family of generators,
 admits small sums and filtered colimits are exact.
 The category of presheaves over any essentially small category $\mathscr S$
 with values in a the category of $R$-modules is a basic example of Grothendieck
 abelian category. In fact, it is generated by representable presheaves of
 $R$-modules. The existence of small sums is obvious
 and the fact filtered colimits are exact can be reduced to the similar fact
 for the category of $R$-modules by taking global sections over objects
 of $\mathscr S$.} 
 It admits a unique symmetric monoidal
 structure such that the Yoneda embedding
$$
\smc \rightarrow \psht,\phantom{i} X \mapsto\prep X
$$
is symmetric monoidal (see e.g. \cite[Lecture 8]{Mazza06}). From the functors \eqref{eq:from_smc2smcV}
 and \eqref{eq:from_sm2smc}, we derive as usual adjunctions of categories:
$$
\xymatrix@=20pt{
\psh\ar@<3pt>^{\tilde \gamma^*}[r] 
 & \psht\ar@<3pt>^{\pi^*}[r]\ar@<1pt>^{\tilde \gamma_*}[l] 
 & \pshtV\ar@<1pt>^{\pi_*}[l]
}
$$
such that $\tilde \gamma_*(F)=F \circ \tilde \gamma$,
 $\pi_*(F)=F \circ \pi$. 
The left adjoints $\tilde \gamma^*$ and $\pi^*$ are easily described as follows. For a smooth scheme $X$, let $R(X)$ be the presheaf (of abelian groups) such that $R(X)(Y)$ is the free $R$-module generated by $\Hom(Y,X)$ for any smooth scheme $Y$. The Yoneda embedding yields $\Hom_{\psh}(R(X),F)=F(X)$ for any presheaf $F$, and in particular 
 \[
 \Hom_{\psh}(R(X),\tilde\gamma_*(F))=F(X)=\Hom_{\psht}(\prepNP_R(X),F)
 \] 
 for any $F\in\psht$. We can thus set $\tilde\gamma^*(R(X))= \prepNP_R(X)$. On the other hand, suppose that 
 \[
 F_1\to F_2\to F_3\to 0
 \]
 is an exact sequence in $\psh$. The functor $\Hom_{\psh}(\_,F)$ being left exact for any $F\in \psh$, we find an exact sequence of presheaves
 \[
 0\to \Hom_{\psh}(F_3,\tilde\gamma_*(G))\to \Hom_{\psh}(F_2,\tilde\gamma_*(G))\to \Hom_{\psh}(F_1,\tilde\gamma_*(G))
 \]
 for any $G\in\psht$ and by adjunction an exact sequence
  \[
 0\to \Hom_{\psh}(\tilde\gamma^*(F_3),G)\to \Hom_{\psh}(\tilde\gamma^*(F_2),G)\to \Hom_{\psh}(\tilde\gamma^*(F_1),G)
 \]
 showing that $\tilde\gamma^*(F_3)$ is determined by $\tilde\gamma^*(F_2)$ and $\tilde\gamma^*(F_1)$, i.e. that the sequence
 \[
  \tilde\gamma^*(F_1)\to \tilde\gamma^*(F_2)\to \tilde\gamma^*(F_3)\to 0
 \]
 is exact. This gives the following formula. If $F$ is a presheaf, we can choose a resolution by (infinite) direct sums of representable presheaves (e.g. \cite[proof of Lemma 8.1]{Mazza06})
 \[
 F_1\to F_2\to F\to 0
 \]
 and compute $\tilde\gamma^*F$ as the cokernel of $\tilde\gamma^*F_1\to \tilde\gamma^*F_2$. We let the reader define $\tilde \gamma^*$ on morphisms and check that it is independent (up to unique isomorphisms) of choices. A similar construction works for $\pi^*$. Note that the left adjoints $\tilde \gamma^*$
 and $\pi^*$ are symmetric monoidal and right-exact.

\begin{lm}\label{lm:compare_transfers}
The functors $\tilde \gamma_*$ and $\pi_*$ are faithful. If $2$ is invertible in the ring $R$ then the functor $\pi_*$ is also full.
\end{lm}

\begin{proof}
The faithfulness of both $\tilde \gamma_*$ and $\pi_*$ are obvious.
 To prove the second assertion, we use the
 fact that the map \eqref{eq:from_corr2corV} from
 finite $\mathrm{MW}$-correspondences to correspondences is surjective
 after inverting $2$ (Lemma \ref{lm:corr&corrV_upto2tor}).
In particular, given a $\mathrm{MW}$-presheaf $F$,
 the property $F=\pi_*(F_0)$ is equivalent to the property on $F$
 that for any $\alpha, \alpha' \in \corr(X,Y)$ with $\pi(\alpha)=\pi(\alpha')$ then $F(\alpha)=F(\alpha')$.
Then it is clear that a natural transformation
 between two presheaves with transfers $F_0$ and $G_0$
 is the same thing as a natural transformation
 between $F_0 \circ \pi$ and $G_0 \circ \pi$.
\end{proof}

\begin{df}
We define a \emph{$\mathrm{MW}$-$t$-sheaf} (resp. $t$-sheaf with transfers)
 to be a presheaf with
 $\mathrm{MW}$-transfers (resp. with transfers) $F$ such that
 $\tilde \gamma_*(F)=F \circ \tilde \gamma$ (resp. $F \circ \gamma$)
 is a sheaf for the given topology $t$.
 When $t$ is the Nisnevich topology, we will simply
 say \emph{$\mathrm{MW}$-sheaf} and when $t$ is the \'etale topology
 we will say \emph{\'etale $\mathrm{MW}$-sheaf}.

We denote by $\sht$ the category of $\mathrm{MW}$-$t$-sheaves,
 seen as a full subcategory of the $R$-linear
 category $\psht$. When $t$ is the Nisnevich topology,
 we drop the index in this notation.
\end{df}
Note that there is an obvious forgetful functor
$$
\tO_t:\sht \rightarrow \psht
 \text{ (resp. } \mathcal O^{\mathrm{tr}}_t:\shtV \rightarrow \pshtV)
$$
which is fully faithful.
In what follows, we will drop
 the indication of the topology $t$ in the above functors,
 as well as their adjoints.

\begin{ex}
Given a smooth scheme $X$, the presheaf $\prep X$ is in general not a $\mathrm{MW}$-sheaf (see \cite[5.12]{Calmes14b}). Note however that $\prep {\spec k}$ is the unramified
 $0$-th Milnor-Witt sheaf $\sKMW_0$ (defined in \cite[\S 3]{Morel08}) by \emph{loc. cit.} Ex. 4.4.
\end{ex}

As in the case of the theory developed by Voevodsky, the theory of 
 $\mathrm{MW}$-sheaves rely on the following fundamental lemma, whose proof
 is adapted from Voevodsky's original argument.
 
\begin{lm}\label{lm:corr_main}
Let $X$ be a smooth scheme and $p:U \rightarrow X$ be a $t$-cover
 where $t$ is the Nisnevich or the \'etale topology.

Then the following complex
\[
\hdots \xrightarrow{\ d_n\ } \prep {U^n_X} \longrightarrow
 \hdots \longrightarrow \prep {U \times_X U} \xrightarrow{\ d_1\ } \prep U
 \xrightarrow{\ d_0\ }\prep X \rightarrow 0
\]
where $d_n$ is the differential associated with the \v Cech simplicial
 scheme of $U/X$, is exact on the associated $t$-sheaves.
\end{lm}

\begin{proof}
We have to prove that the fiber of the above complex at a $t$-point
 is an acyclic complex of $R$-modules.  Taking into
 account Remark \ref{rem:finite_corr&plim}(4), we are reduced to 
 prove, given an essentially smooth local henselian scheme $S$, that
 the following complex
$$
C_*:=\hdots \xrightarrow{\ d_n\ } \corr(S,U^n_X) \longrightarrow
 \hdots \longrightarrow \corr(S,U \times_X U) \xrightarrow{\ d_1\ } \corr(S,U)
 \xrightarrow{\ d_0\ } \corr(S,X) \rightarrow 0
$$
is acyclic.

Let $\mathcal A=\mathcal A(S,X)$ be the set of admissible subsets in $S\times X$ (\cite[Definition 4.1]{Calmes14b}). 
 Given any $T \in \mathcal A$, and an integer $n \geq 0$,
 we let $C_n^{(T)}$ be the subgroup of $\corr(S,U^n_X)$ consisting of
 $\mathrm{MW}$-correspondences whose support is in the closed subset $U^n_T:=T \times_X U^n_X$ of $S\times U^n_X$.
 The differentials are given by direct images along projections
 so they respect the support condition on $\mathrm{MW}$-correspondence
 associated with $T \in \mathcal F$ and make $C_*^{(T)}$
 into a subcomplex of $C_*$.

It is clear that $C_*$ is the filtering union of the subcomplexes
 $C_*^{(T)}$ for $T \in \mathcal F$ so it suffices to prove that,
 for a given $T \in \mathcal F$, the complex $C_*^{(T)}$ is split. We prove the result when $R=\ZZ$, the general statement follows after tensoring with $R$.
 Because $S$ is henselian and $T$ is finite over $S$, the scheme $T$
 is a finite sum of local henselian schemes. Consequently,
 the $t$-cover $p_T:U_T \rightarrow T$, which is in particular
 \'etale and surjective, admits a splitting $s$. It follows from \cite[Proposition 2.15]{Milne12} that $s$ is an isomorphism onto a connected component of $U_T$.
 We therefore obtain maps $s \times 1_{U_T^n}:U_T^n\to U^{n+1}_T$ such that $U^{n+1}_T=U_T^n\sqcup D^{n+1}_{T}$ for any $n\geq 0$ and a commutative diagram
 \[
 \xymatrix{U^n_T\ar[d]\ar[r] & (S\times U_X^{n+1})\setminus D^{n+1}_{T}\ar[d] \\
 U^{n+1}_T\ar[r]\ar[d] & S\times U_X^{n+1}\ar[d] \\
 U_T^n\ar[r] & S\times U_X^{n}}
 \]
in which the squares are Cartesian and the right-hand vertical maps are \'etale. By \'etale excision, we get isomorphisms
\[
\wCH^*_{U_T^n}(S\times U_X^{n},\omega_{U_X^{n}})\to \wCH^*_{U_T^n}((S\times U_X^{n+1})\setminus D^{n+1}_T,\omega_{U_X^{n+1}})
\]
and
\[
\wCH^*_{U_T^n}(S\times U_X^{n+1},\omega_{U_X^{n+1}})\to \wCH^*_{U_T^n}((S\times U_X^{n+1})\setminus D^{n+1}_T,\omega_{U_X^{n+1}}).
\]
Putting these isomorphisms together, we obtain an isomorphism
\[
\wCH^*_{U_T^n}(S\times U_X^{n},\omega_{U_X^{n}})\to \wCH^*_{U_T^n}(S\times U_X^{n+1},\omega_{U_X^{n+1}})
\]
that we can compose with the extension of support to finally obtain a homomorphism
\[
(s \times 1_{U_T^n})_*:\wCH^*_{U_T^n}(S\times U_X^{n},\omega_{U_X^{n}})\to \wCH^*_{U_T^{n+1}}(S\times U_X^{n+1},\omega_{U_X^{n+1}})
\]
yielding a contracting homotopy
\[
(s \times 1_{U_T^n})_*:C_n^{(T)} \rightarrow C_{n+1}^{(T)}.
\]
\end{proof}

\begin{num}\label{num:rightadjoint}
As in the classical case, one can derive from this lemma
 the existence of a left adjoint $\tilde a$ to the functor $\tO$.
 The proof is exactly the same as in the case of sheaves
 with transfers (cf. \cite[10.3.9]{CD3} for example) but we include it here for the convenience
 of the reader.

Let us introduce a notation. If $P$ is a presheaf on $\sm$,
 we define a presheaf with transfers
\begin{equation}\label{eq:df_gamma^!}
\tilde\gamma^!(P):Y \mapsto \Hom_{\psh}(\tilde\gamma_*(\prep Y),P).
\end{equation}
and we observe that  $\tilde\gamma^!$ is right adjoint to the functor $\tilde\gamma_*$. The latter, having both a left and a right adjoint, is then exact.
Given a natural transformation
$$
\phi:P \rightarrow \tilde\gamma_*\tilde\gamma^!(P)
$$
and smooth schemes $X$ and $Y$, 
we define a pairing
$$
P(X) \times \corr(Y,X) \rightarrow P(Y),
 (\rho,\alpha) \mapsto \langle \rho,\alpha \rangle_\phi
 :=[\phi_X(\rho)]_Y(\alpha)
$$
where $\phi_X(\rho)$ is seen as a natural transformation
 $\prep X \rightarrow P$.
 The following lemma is tautological.
\end{num}

\begin{lm}\label{lm:W-transfers_basic_existence}
Let $P$ be a presheaf on $\sm$.
 Then there is a bijection between the following data:
\begin{itemize}
\item Presheaves with transfers $\tilde P$ such that $\tilde\gamma_*(\tilde P)=P$;
\item Natural transformations $\phi:P \rightarrow \tilde\gamma_*\tilde\gamma^!(P)$
 such that:
\begin{enumerate}
\item[(W1)] $\forall \rho \in P(X), \langle \rho, Id_X \rangle_\phi=\rho$.
\item[(W2)] $\forall (\rho,\beta,\alpha)
 \in P(X) \times \corr(Y,X) \times \corr(Z,Y),
  \langle\langle \rho,\beta \rangle_\phi,\alpha\rangle_\phi
	 =\langle \rho,\beta \circ \alpha\rangle_\phi$;
\end{enumerate}
\end{itemize}
 according to the following rules:
\begin{align*}
\tilde P & \mapsto \big(P=\tilde\gamma_*(\tilde P) \xrightarrow{ad'}
 \tilde\gamma_*\tilde\gamma^!\tilde\gamma_*(\tilde P)=\tilde\gamma_*\tilde\gamma^!(P)\big) \\
(P,\langle.,\alpha\rangle_\phi) & \mapsfrom \phi,
\end{align*}
where $ad'$ is the unit map for the adjunction $(\tilde\gamma_*,\tilde\gamma^!)$. 
\end{lm}
Before going further, we note the following corollary of the previous result.
 
\begin{cor}\label{cor:prop:corr_main}
\begin{enumerate}
\item For any $t$-sheaf $F$ on $\sm$,
 $\tilde\gamma^!(F)$ is a $\mathrm{MW}$-$t$-sheaf.
\item Let $\alpha \in \corr(X,Y)$ be a finite $\mathrm{MW}$-correspondence
 and $p:W \rightarrow Y$ a $t$-cover.
 Then there exists a $t$-cover $q:W' \rightarrow X$
 and a finite $\mathrm{MW}$-correspondence $\hat \alpha:W' \rightarrow W$
 such that the following diagram commutes:
\begin{equation}\label{eq:lift_corr}
\begin{split}
\xymatrix@=18pt{
W'\ar^-{\hat \alpha}[r]\ar|/-8pt/\bullet[r]\ar_q[d] & W\ar^p[d] \\
X\ar_-\alpha[r]\ar|/-10pt/\bullet[r] & Y
}
\end{split}
\end{equation}
\end{enumerate}
\end{cor}

The first property is a direct consequence of Lemma
 \ref{lm:corr_main} given Formula \eqref{eq:df_gamma^!}.
 The second property is an application of the
 fact that $\prep W \rightarrow \prep X$ is an epimorphism
 of sheaves, obtained from the same proposition.

We are ready to state and prove the main lemma
 which proves the existence of the right adjoint $\tilde a$ to $\tO$.
 
\begin{lm}\label{lm:exists-Wtr}
Let $\tilde P$ be a $\mathrm{MW}$-presheaf and $P:=\tilde\gamma_*(\tilde P)$. Let $F$
be the $t$-sheaf associated with $P$ and let 
 $\tau:P \rightarrow F$ be the canonical natural transformation.

Then there exists a unique pair $(\tilde F,\tilde \tau)$
 such that:
\begin{enumerate}
\item $\tilde F$ is a $\mathrm{MW}$-$t$-sheaf such that $\tilde\gamma_*(\tilde F)=F$.
\item $\tilde \tau:\tilde P \rightarrow \tilde F$ is a natural transformation
 of $\mathrm{MW}$-presheaves such that the induced transformation
$$
P=\tilde\gamma_*(\tilde P) \xrightarrow{\tilde\gamma_*(\tilde \tau)} \tilde\gamma_*(\tilde F)=F
$$
coincides with $\tau$.
\end{enumerate}
\end{lm}
\begin{proof}
Let us construct
 $\tilde F$. Applying Lemma \ref{lm:W-transfers_basic_existence} to $\tilde P$ and $P$, we get a 
 canonical natural transformation: $\psi:P \rightarrow \tilde\gamma_*\tilde\gamma^!(P)$.
 Applying point (1) of Corollary \ref{cor:prop:corr_main} and the fact that $F$ is the $t$-sheaf associated with $P$,
 there exists a unique natural transformation $\phi$ which
 fits into the following commutative diagram:
$$
\xymatrix@=20pt{
P\ar^-{\psi}[r]\ar_\tau[d]
 & \tilde\gamma_*\tilde\gamma^!(P)\ar^{\tilde\gamma_*\tilde\gamma^!(\tau)}[d] \\
F\ar_-\phi[r] & \tilde\gamma_*\tilde\gamma^!(F).
}
$$
To obtain the $\mathrm{MW}$-sheaf $\tilde F$ satisfying (1),
 it is sufficient according to
 Lemma \ref{lm:W-transfers_basic_existence} to prove
 that conditions (W1) and (W2) are satisfied for the
 pairing $\langle.,.\rangle_\phi$. Before proving this,
 we note that the existence of $\tilde \tau$ satisfying property (2)
 is equivalent to the commutativity of the above diagram.
 In particular, the unicity of $(\tilde F,\tilde \tau)$ comes
 from the unicity of the map $\phi$.

Therefore, we only need to prove (W1) and (W2) for $\phi$.
Consider a couple $(\rho,\alpha) \in F(X) \times \corr(Y,X)$.
Because $F$ is the $t$-sheaf associated with $P$,
 there exists a $t$-cover $p:W \rightarrow X$ 
 and a section $\hat \rho \in P(W)$
 such that $p^*(\rho)=\tau_W(\hat \rho)$.
 According to point (2) of Corollary \ref{cor:prop:corr_main},
 we get a $t$-cover $q:W' \rightarrow Y$
 and a correspondence $\hat \alpha \in \corr(W',W)$
  making the diagram \eqref{eq:lift_corr} commutative.
As $\phi$ is a natural transformation, we get
$$
q^*\langle \rho,\alpha \rangle_\phi
=\langle \rho,\alpha \circ q\rangle_\phi
=\langle \rho,p \circ \hat \alpha \rangle_\phi
=\langle p^*\rho,\hat \alpha \rangle_\phi
=\langle \tau_W(\hat \rho),\hat \alpha \rangle_\phi
=\langle \hat \rho,\hat \alpha \rangle_\psi.
$$
Because $q^*:F(X) \rightarrow F(W)$ is injective,
 we deduce easily from this principle the properties (W1) and (W2)
 for $\phi$ from their analog properties for $\psi$. 
\end{proof}

\begin{prop}\label{prop:exist_associated-W-t-sheaf}
\begin{enumerate}
\item The obvious forgetful functor $\tO:\sht \rightarrow \psht$ admits
 a left adjoint $\tilde a$ such that the following diagram commutes:
$$
\xymatrix@=12pt@C=16pt{
\psh\ar_a[d] & \psht\ar^-{\tilde a}[d]\ar_-{\tilde \gamma_*}[l] \\
\sh & \sht\ar_-{\tilde \gamma_*}[l]
}
$$
where $a$ is the usual $t$-sheafification functor with
 respect to the smooth site.
\item The category $\sht$ is a Grothendieck abelian category
 and the functor $\tilde a$ is exact.
\item The functor $\tilde \gamma_*$, appearing in the lower line
 of the preceding diagram, admits a left adjoint $\tilde \gamma^*$,
 and commutes with every limits and colimits.
\end{enumerate}
\end{prop}

\begin{proof}
The first point follows directly from the previous lemma:
 indeed,  with the notation of this lemma, we can put:
 $\tilde a(P)=\tilde F$.

For point (2), we first remark that the functor $\tilde a$,
 being a left adjoint, commutes with every colimits. Moreover,
the functor $a$ is exact and $\tilde\gamma_*:\psht\to \psh$ is also exact (Paragraph \ref{num:rightadjoint}).
 Therefore, $\tilde a$ is exact because of the previous commutative square
 and the fact that $\tilde \gamma_*$ is faithful. Then, we easily deduce that $\sht$ is a Grothendieck abelian category
 from the fact that $\psht$ is such a category.

The existence of the left adjoint $\tilde \gamma^*$ follows
 formally. Thus $\tilde \gamma_*$ commutes with every limits.
 Because  $\tilde \gamma_*$ is exact and commutes with
 arbitrary coproducts, we deduce that it commutes with arbitrary
 colimits, therefore proving point (3).
\end{proof}

\begin{rem}
The left adjoint $\tilde \gamma^*$ of $\tilde \gamma_*:\sht\to \sh$ can be computed as the composite
\[
\sh\stackrel{\mathcal O}\to \psh\stackrel{\tilde \gamma^*}\to \psht\stackrel{\tilde a}\to \sht.
\] 
One can also observe that, according to point (2), a family of generators of
 the Grothendieck abelian category $\sht$ is obtained by applying
 the functor $\tilde a$ to a family of generators of $\psht$. 
\end{rem}
 
\begin{df}\label{df:gen_sht}
Given any smooth scheme $X$, we put $\repR X=\tilde a\left(\prep X\right)$. 
\end{df}

In particular, for a smooth scheme $X$, $\tilde R_t(X)$
 is the $t$-sheaf associated with the presheaf $\prep X$,
 equipped with its canonical action of $\mathrm{MW}$-correspondences
 (Lemma \ref{lm:exists-Wtr}). The corresponding family, 
 for all smooth schemes $X$, generates the abelian category $\sht$.

\begin{num}\label{num:sht_monoidal}
One deduces from the monoidal structure on $\smc$ a monoidal
 structure on $\sht$ whose tensor product $\otr$ is uniquely characterized
 by the property that for any smooth schemes $X$ and $Y$:
\begin{equation}\label{eq:sht_monoidal}
\repR X \otr \repR Y=\repR {X\times Y}.
\end{equation}
Explicitly, the tensor product of any two sheaves $F,G\in \sht$ is obtained by applying $\tilde a$ to the presheaf tensor product $F\otimes G$ mentioned after Definition \ref{def:representablepsht}.
In particular, the bifunctor $\otr$ commutes with colimits
 and therefore, as the abelian category $\sht$ is a Grothendieck abelian
 category, the monoidal category $\sht$ is closed. The internal Hom
 functor is characterized by the property that for any $\mathrm{MW}$-$t$-sheaf $F$
 and any smooth scheme $X$,
$$
\uHom\big(\repR X,F\big)=F(X \times -).
$$
\end{num}

As a corollary of Proposition \ref{prop:exist_associated-W-t-sheaf},
 we obtain functors between the category of sheaves we have
 considered so far.
\begin{cor}\label{cor:adjunctions_corr}
\begin{enumerate}
\item There exists a commutative diagram of symmetric monoidal functors
$$
\xymatrix@=26pt{
\psh\ar_{\tilde \gamma^*}[d]\ar^{a_\nis}[r]
 & \shx{\nis}\ar^{\tilde \gamma^*_\nis}[d]\ar^{a_\et}[r]
  & \shx{\et}\ar^{\tilde \gamma^*_\et}[d] \\
\psht\ar_{\pi^*}[d]\ar^{\tilde a_\nis}[r]
 & \shtx{\nis}\ar^{\pi^*_\nis}[d]\ar^{\tilde a_\et}[r]
 & \shtx{\et}\ar^{\pi^*_\et}[d] \\
\pshtV\ar^{a^{tr}_\nis}[r] & \shtVx{\nis}\ar^{a^{tr}_\et}[r]
 & \shtVx{\et}
}
$$
which are all left adjoints of an obvious forgetful functor. Each of these functors
 respects the canonical family of abelian generators.
\item Let $t=\nis, \et$. 
Then the right adjoint functor $\tilde \gamma_*^t:\sht \rightarrow \sh$
 is faithful. If $2$ is invertible in $R$,
 the right adjoint functor $\pi_*^t:\shtV \rightarrow \sht$ is fully faithful.
\end{enumerate}
\end{cor}
Indeed, the first point is a formal consequence of Proposition 
 \ref{prop:exist_associated-W-t-sheaf} and its analog for sheaves with transfers.
 The second point follows from the commutativity of the diagram in point (1),
 which induces an obvious commutative diagram for the right adjoint functors,
 the fact that the forgetful functor from sheaves to presheaves is always
 fully faithful and Lemma \ref{lm:compare_transfers}.

\section{Framed correspondences}\label{sec:framed}

\subsection{Definitions and basic properties}

The aim of this section is to make a link between the category of linear framed correspondences (after Garkusha-Panin-Voevodsky) and the category of $\mathrm{MW}$-presheaves. We start with a quick reminder on framed correspondences following \cite{Garkusha14}.

\begin{df}
Let $U$ be a smooth $k$-scheme and $Z\subset U$ be a closed subset of codimension $n$. A set of regular functions $\phi_1,\ldots,\phi_n\in k[U]$ is called a framing of $Z$ in $U$ if $Z$ coincides with the closed subset $\phi_1=\ldots=\phi_n=0$. 
\end{df} 

\begin{df}
Let $X$ and $Y$ be smooth $k$-schemes, and let $n\in\NN$ be an integer. An explicit framed correspondence $c=(U,\phi,f)$ of level $n$ from $X$ to $Y$ consists of the following data:
\begin{enumerate}
\item A closed subset $Z\subset \AA^n_X$ which is finite over $X$ (here, $Z$ is endowed with its reduced structure). 
\item An \'etale neighborhood $\alpha:U\to \AA^n_X$ of $Z$.
\item A framing $\phi=(\phi_1,\ldots,\phi_n)$ of $Z$ in $U$.
\item A morphism $f:U\to Y$.
\end{enumerate} 
The closed subset $Z$ is called the \emph{support} of the explicit framed correspondence $c=(U,\phi,f)$.
\end{df}

\begin{rem}
One could give an alternative approach to the above definition. A framed correspondence $(U,\phi,f)$ corresponds to a pair of morphisms $\phi:U\to \AA^n_k$ and $f:U\to Y$ yielding a unique morphism $\varphi:U\to \AA^n_Y$. The closed subset $Z\subset U$ corresponds to the preimage of $Y\times \{0\}\subset Y\times \AA^n_k=\AA^n_Y$. This correspondence is unique.
\end{rem}

\begin{rem}
Note that $Z$ is not supposed to map surjectively onto a component of $X$. For instance $Z=\emptyset$ is an explicit framed correspondence of level $n$, denoted by $0_n$. If $Z$ is non-empty, then an easy dimension count shows that $Z\subset \AA^n_X\to X$ is indeed surjective onto a component of $X$.
\end{rem}

\begin{rem}
Suppose that $X$ is a smooth connected scheme. By definition, an explicit framed correspondence of level $n=0$ is either a morphism of schemes $f:X\to Y$ or $0_0$.
\end{rem}

\begin{df}
Let $c=(U,\phi,f)$ and $c^\prime=(U^\prime,\phi^\prime,f^\prime)$ be two explicit framed correspondences of level $n\geq 0$. Then, $c$ and $c^\prime$ are said to be \emph{equivalent} if they have the same support and  there exists an open neighborhood $V$ of $Z$ in $U\times_{\AA^n_X} U^\prime$ such that the diagrams
\[
\xymatrix{U\times_{\AA^n_X} U^\prime\ar[r]\ar[d] & U^\prime\ar[d]^-{f^\prime} \\
U\ar[r]_-f & Y}
\]
and
\[
\xymatrix{U\times_{\AA^n_X} U^\prime\ar[r]\ar[d] & U^\prime\ar[d]^-{\phi^\prime} \\
U\ar[r]_-{\phi} & \AA^n_k}
\]
are both commutative when restricted to $V$. A \emph{framed} correspondence of level $n$ is an equivalence class of explicit framed correspondences of level $n$.
\end{df}

\begin{df}
Let $X$ and $Y$ be smooth schemes and let $n\in \NN$. We denote by $\mathrm{Fr}_n(X,Y)$ the set of framed correspondences of level $n$ from $X$ to $Y$ and by $\mathrm{Fr}_*(X,Y)$ the set $\sqcup_n \mathrm{Fr}_n(X,Y)$. Together with the composition of framed correspondences described in \cite[\S 2]{Garkusha14}, this defines a category whose objects are smooth schemes and morphisms are $\mathrm{Fr}_*(\_,\_)$. We denote this category by $\mathrm{Fr}_*(k)$ and refer to it as the \emph{category of framed correspondences}.
\end{df}

We now pass to the linear version of the above category following \cite[\S 7]{Garkusha14}, starting with the following observation. Let $X$ and $Y$ be smooth schemes, and let $c_{Z}=(U,\phi,f)$ be an explicit framed correspondence of level $n$ from $X$ to $Y$ with support $Z$ of the form $Z=Z_1\sqcup Z_2$. Let $U_1=U\setminus Z_2$ and $U_2=U\setminus Z_1$. For $i=1,2$, we get \'etale morphisms $\alpha_i:U_i\to X$ and morphisms $\phi_i:U_i\to \AA^n_k$, $f_i:U_i\to Y$ by precomposing the morphisms $\alpha,\phi$ and $f$ with the open immersion $U_i\to U$. Note that $U_i$ is an \'etale neighborhood of $Z_i$ for $i=1,2$ and that $c_{Z_i}=(U_i,\phi_i,f_i)$ are explicit framed correspondences of level $n$ from $X$ to $Y$ with support $Z_i$.

\begin{df}\label{def:linear}
Let $X$ and $Y$ be smooth schemes and let $n\in\NN$. Let 
\[
\ZZ \mathrm{F}_n(X,Y)=\ZZ \mathrm{Fr}_n(X,Y)/H
\]
where $H$ is the subgroup generated by elements of the form $c_Z-c_{Z_1}-c_{Z_2}$ where $Z=Z_1\sqcup Z_2$ is as above and $\ZZ \mathrm{Fr}_n(X,Y)$ is the free abelian group on $\mathrm{Fr}_n(X,Y)$. The category $\ZZ \mathrm{F}_*(k)$ of \emph{linear framed correspondences} is the category whose objects are smooth schemes and whose morphisms are 
\[
\mathrm{Hom}_{\ZZ \mathrm{F}_*(k)}(X,Y)=\bigoplus_{n\in \NN} \ZZ \mathrm{Fr}_n(X,Y).
\]
\end{df}

\begin{rem}
Note that there is an obvious functor $\iota:\mathrm{Fr}_*(k)\to \ZZ \mathrm{F}_*(k)$ with $\iota(0_n)=0$ for any $n\in\NN$.
\end{rem}

The stage being set, we now compare the category of finite $\mathrm{MW}$-correspondences with the above categories.

Let $U$ be a smooth $k$-scheme and let $\phi:U\to \AA^n_k$ be a morphism corresponding to (nonzero) global sections $\phi_i\in \OO(U)$. Each section $\phi_i$ can be seen as an element of $k(U)^\times$ and defines then an element of $\KMW_1(k(U))$. Let $\vert \phi_i\vert$ be the support of $f_i$, i.e. its vanishing locus, and let $Z= \vert \phi_1\vert \cap \ldots \cap \vert \phi_n\vert $. Consider the residue map
\[
d:\KMW_1(k(U))\to \bigoplus_{x\in U^{(1)}} \KMW_0(k(x),\omega_x).
\]
Then, $d(\phi_i)$ defines an element supported on $\vert \phi_i\vert$. As it is a boundary, it defines a cycle $Z(\phi_i)\in H^1_{\vert \phi_i\vert}(U,\sKMW_1)$. Now, we can consider the intersection product
\[
\H^1_{\vert \phi_1\vert}(U,\sKMW_1)\times \ldots \times \H^1_{\vert \phi_n\vert}(U,\sKMW_1)\to \H^n_{Z}(U,\sKMW_n)
\]
to get an element $Z(\phi_1)\cdot \ldots \cdot Z(\phi_n)$ that we denote by $Z(\phi)$.

\begin{lm}\label{lem:functor}
Any explicit framed correspondence $c=(U,\phi,f)$ induces a finite $\mathrm{MW}$-correspondence $\alpha(c)$ from $X$ to $Y$. Moreover, two equivalent explicit framed correspondences $c$ and $c^\prime$ induce the same finite $\mathrm{MW}$-correspondence.
\end{lm}

\begin{proof}
Let us start with the first assertion. If $Z$ is empty, its image is defined to be zero. If $c$ is of level $0$, then it corresponds to a morphism of schemes and we use the functor $\sm\to \smc $ to define the image of $c$. We thus suppose that $Z$ is non-empty (thus finite and surjective on some components of $X$) of level $n\geq 1$.
Consider the following diagram
\[
\xymatrix{ & U\ar[d]_-\alpha\ar[r]^-{(\phi,f)} & \AA^n_Y \\
Z\ar[r]\ar[ru] &  \AA^n_X\ar[d]_-{p_X} & \\
 & X & }
\]
defining an explicit framed correspondence $(U,\phi,f)$ of level $n$. The framing $\phi$ defines an element $Z(\phi)\in \H^n_{Z}(U,\sKMW_n)$ as explained above. Now, $\alpha$ is \'etale and therefore induces an isomorphism  $\alpha^*\omega_{\AA^n_X}\simeq \omega_U$. Choosing the usual orientation for $\AA^n_k$, we get an isomorphism $\OO_{\AA^n_X}\simeq \omega_{\AA^n_X}\otimes (p_X)^*\omega_X^\vee$ and therefore an isomorphism
\[
\OO_U\simeq \alpha^*(\OO_{\AA^n_X})\simeq \alpha^*(\omega_{\AA^n_X}\otimes p_X^*\omega_X^\vee)\simeq \omega_U\otimes (p_X\alpha)^* \omega_X^\vee.
\]
We can then see $Z(\phi)$ as an element of the group $\H^n_{Z}(U,\sKMW_n,\omega_U\otimes (p_X\alpha)^* \omega_X^\vee)$. Consider next the map $(p_X\alpha,f):U\to X\times Y$ and the image $T$ of $Z$ under the map of underlying topological spaces. It follows from \cite[Lemma 1.4]{Mazza06} that $T$ is closed, finite and surjective over (some components of) $X$. Moreover, the morphism $Z\to T$ is finite and it follows that we have a push-forward homomorphism
\[
(p_X\alpha,f)_*:\H^n_{Z}(U,\sKMW_n,\omega_U\otimes (p_X\alpha)^* \omega_X^\vee)\to \H^n_T(X\times Y,\sKMW_n,\omega_{X\times Y/X})
\] 
yielding, together with the canonical isomorphism $\omega_{X\times Y/X}\simeq \omega_Y$, a finite Chow-Witt correspondence $\alpha(c):=(p_X\alpha,f)_*(Z(\phi))$ between $X$ and $Y$.

Suppose next that $c=(U,\phi,f)$ and $c^\prime=(U^\prime,\phi^\prime,f^\prime)$ are two equivalent explicit framed correspondences of level $n$. Following the above construction, we obtain two cocycles $\tilde\alpha(c)\in \H^n_{Z}(U,\sKMW_n,\omega_U\otimes (p_X\alpha)^* \omega_X^\vee)$ and $\tilde\alpha(c^\prime)\in \H^n_{Z}(U^\prime,\sKMW_n,\omega_{U^\prime}\otimes (p_X\alpha^\prime)^* \omega_X^\vee)$. Now, the pull-backs along the projections 
\[
\xymatrix{U\times_{\AA^n_X} U^\prime\ar[r]^-{p_2}\ar[d]_-{p_1} & U^\prime \\
U & }
\]
yield homomorphisms
\[
p_1^*:\H^n_{Z}(U,\sKMW_n,\omega_U\otimes (p_X\alpha)^* \omega_X^\vee)\simeq \H^n_{p_1^{-1}(Z)}(U\times_{\AA^n_X} U^\prime,\sKMW_n,\omega_{U\times_{\AA^n_X} U^\prime}\otimes (p_X\alpha p_1)^* \omega_X^\vee)
\] 
and
\[
p_2^*:\H^n_{Z}(U^\prime,\sKMW_n,\omega_{U^\prime}\otimes (p_X\alpha^\prime)^* \omega_X^\vee)\simeq \H^n_{p_2^{-1}(Z)}(U\times_{\AA^n_X} U^\prime,\sKMW_n,\omega_{U\times_{\AA^n_X} U^\prime}\otimes (p_X\alpha p_2)^* \omega_X^\vee),
\] 
while the pull-back along the open immersion $i:V\to U\times_{\AA^n_X} U^\prime$ induces homomorphisms 
\[
i^*:\H^n_{p_1^{-1}(Z)}(U\times_{\AA^n_X} U^\prime,\sKMW_n,\omega_{U\times_{\AA^n_X} U^\prime}\otimes (p_X\alpha p_1)^* \omega_X^\vee)\simeq \H^n_Z(V,\sKMW_n, \omega_V\otimes (p_X\alpha p_1i)^* \omega_X^\vee)
\]
and
\[
i^*:\H^n_{p_2^{-1}(Z)}(U\times_{\AA^n_X} U^\prime,\sKMW_n,\omega_{U\times_{\AA^n_X} U^\prime}\otimes (p_X\alpha p_2)^* \omega_X^\vee)\simeq \H^n_Z(V,\sKMW_n, \omega_V\otimes (p_X\alpha p_2i)^* \omega_X^\vee).
\]
Note that $p_X\alpha p_2=p_X\alpha p_1$ and that $i^*p_1^*(\tilde\alpha(c))=i^*p_2^*(\tilde\alpha(c^\prime))$ by construction. Pushing forward along $V\to U\times_{\AA^n_X} U^\prime\to U\to X\times Y$, we get the result.
\end{proof}

\begin{ex}\label{ex:stable}
Let $X$ be a smooth $k$-scheme. Consider the explicit framed correspondence $\sigma_X$ of level $1$ from $X$ to $X$ given by $(\AA^1_X,q,p_X)$ where $q:\AA^1_X=\AA^1\times X\to \AA^1$ is the projection to the first factor and $p_X:\AA^1_X\to X$ is the projection to the second factor. We claim that $\alpha(\sigma_X)=Id\in \smc (X,X)$. To see this, observe that we have a commutative diagram
\[
\xymatrix{\AA^1_X\ar[r]^-{(p_X,p_X)}\ar[d]_-{p_X} & X\times X \\
X\ar[ru]_-{\triangle} & }
\]
where $\triangle$ is the diagonal map. Following the process of the above lemma, we start by observing that $Z(q)\in \H^1_X(\AA^1_X,\sKMW_1)$ is the class of $\langle 1\rangle\otimes \overline t\in \sKMW_0(k(X),(\mathfrak m/\mathfrak m^2)^*)$ where $\mathfrak m$ is the maximal ideal corresponding to $X$ in the appropriate local ring and $t$ is a coordinate of $\AA^1$. Now, we choose the canonical orientation of $\AA^1$ and the class of $Z(q)$ corresponds then to the class of $\langle 1\rangle \in \sKMW_0(k(X))$ in $\H^1_X(\AA^1_X,\sKMW_1,\omega _{(\AA^1_X/X)})$. Its push-forward under 
\[
(p_X)_*:\H^1_X(\AA^1_X,\sKMW_1,\omega _{(\AA^1_X/X)})\to \H^0(X,\sKMW_0)
\]
is the class of $\langle 1\rangle$ and the claim follows from the fact that $(p_X,p_X)_*=\triangle_*(p_X)_*$ and the definition of the identity in $\smc (X,X)$.
\end{ex}

\begin{prop}\label{prop:bunchoffunctors}
The assignment $c=(U,\phi,f)\mapsto \alpha(c)$ made explicit in Lemma \ref{lem:functor} define functors $\alpha:\mathrm{Fr}_*(k)\to \smc $ and $\alpha^\prime:\ZZ \mathrm{F}_*(k)\to \smc$ such that we have a commutative diagram of functors
\[
\xymatrix{ &  \mathrm{Fr}_*(k)\ar[dd]_-{\alpha}\ar[rd]^-{\iota} & \\ \sm\ar[ru]\ar[rd]_-{\tilde\gamma} &  & \ZZ \mathrm{F}_*(k)\ar[ld]^-{\alpha^\prime} \\ & \smc . & }
\]
\end{prop}

\begin{proof}
For any smooth schemes $X,Y$ and any integer $n\geq 0$, we have a well-defined map $\alpha:\mathrm{Fr}_n(X,Y)\to \smc(X,Y)$ and therefore a well-defined map $\ZZ \mathrm{Fr}_n(X,Y)\to \smc(X,Y)$. Let $c=(U,\phi,f)$ be an explicit framed correspondence of level $n$ with support $Z$ of the form $Z=Z_1\sqcup Z_2$. Let $c_i=(U_i,\phi_i,f_i)$ be the explicit framed correspondences with support $Z_i$ obtained as in Definition \ref{def:linear}. By construction, we get $\alpha(c)=\alpha(c_1)+\alpha(c_2)$ and it follows that $\alpha:\mathrm{Fr}_n(X,Y)\to \smc(X,Y)$ induces a homomorphism $\alpha^\prime:\ZZ \mathrm{F}_n(X,Y)\to \smc(X,Y)$.

 It remains then to show that the functors $\alpha:\mathrm{Fr}_k\to \smc$ and $\alpha^\prime:\ZZ \mathrm{F}_*(k)\to \smc$ are well-defined, which amounts to prove that the respective compositions are preserved. Suppose then that $(U,\phi,f)$ is an explicit framed correspondence of level $n$ between $X$ and $Y$, and that $(V,\psi,g)$ is an explicit framed correspondence of level $m$ between $Y$ and $Z$. We use the diagram
\begin{equation}\label{eqn:diag1}
\xymatrix{W\ar[r]^-{pr_V}\ar[d]\ar@/_3pc/[dd]_-{pr_U} & V\ar[d]_-\beta\ar[r]^-{\psi}\ar@/_/[rrd]_-g & \AA^m &  \\
U\times \AA^m\ar[r]_-{f\times Id}\ar[d] & Y\times \AA^m\ar[d]_-{p_Y} & & Z \\
U\ar[r]_-f\ar[d]_-{p_X\alpha}\ar@/_1pc/[rrd]_-\phi & Y & &  \\
X & & \AA^n & }
\end{equation} 
%
%
%
in which the squares are all cartesian. The composition of $(U,\phi,f)$ with $(V,\psi,g)$ is given by $(W,(\phi\circ pr_U,\psi\circ pr_V),g\circ pr_V)$. 

On the other hand, the morphisms $(p_X\alpha,f)\circ pr_U:W\to X\times Y$ and $(p_Y\beta,g)\circ pr_V:W\to Y\times Z$ yield a morphism $\rho:W\to X\times Y\times Z$ and then a diagram
\begin{equation}\label{eqn:diag3}
\xymatrix{ & W\ar[r]^-{pr_V}\ar[d]_-\rho & V\ar[d]^-{(p_Y\beta,g)} \\ 
W\ar[r]^-\rho\ar[d]_-{pr_U} & X\times Y\times Z\ar[d]_-{p_{X\times Y}}\ar[r]^-{p_{Y\times Z}} & Y\times Z\ar[d] \\
U\ar[r]_-{(p_X\alpha,f)} & X\times Y\ar[r] & Y}
\end{equation}
in which all squares are cartesian. By base change (\cite[Proposition 3.2, Remark 3.3]{Calmes14b}), we have $(p_{X\times Y})^*(p_X\alpha,f)_*=\rho_*(pr_U)^*$ and $(p_{Y\times Z})^*(p_Y\beta,g)_*=\rho_*(pr_V)^*$.  By definition of the pull-back and the product, we have $(pr_U)^*(Z(\phi))=Z(\phi\circ pr_U)$ and $(pr_V)^*(Z(\psi))=Z(\psi\circ pr_V)$. It follows that 
\[
Z(\phi\circ pr_U,\psi\circ pr_V)=(pr_U)^*(Z(\phi))\cdot (pr_V)^*(Z(\psi)).
\]
Finally, observe that there is a commutative diagram
\[
\xymatrix{W\ar@{=}[d]\ar[rr]^-{g\circ pr_V} & & Z \\
W\ar[r]^-\rho\ar@{=}[d] & X\times Y\times Z\ar[r]^-{p_{X\times Z}} & X\times Z\ar[d]\ar[u] \\
W\ar[rr]_-{p_X\circ \alpha\circ pr_U} & & X.}
\]
Using these ingredients, we see that the composition is preserved.
\end{proof}

\begin{rem}
Note that the functor $\alpha^\prime:\ZZ \mathrm{F}_*(k)\to \smc$ is additive. It follows from Example \ref{ex:stable} that it is not faithful.
\end{rem}

\subsection{Presheaves}

Let $X$ be a smooth scheme. Recall from Example \ref{ex:stable} that we have for any smooth scheme $X$ an explicit framed correspondence $\sigma_X$ of level $1$ given by the triple $(\AA^1_X,q,p_X)$ where $q$ and $p_X$ are respectively the projections onto $\AA^1_k$ and $X$. The following definition can be found in \cite[\S 1]{Garkusha15}.

\begin{df}
Let $R$ be a ring. A presheaf of $R$-modules $F$ on $\ZZ \mathrm{F}_*(k)$ is \emph{quasi-stable} if for any smooth scheme $X$, the pull-back map $F(\sigma_X):F(X)\to F(X)$ is an isomorphism. A quasi-stable presheaf is \emph{stable} if $F(\sigma_X):F(X)\to F(X)$ is the identity map for any $X$. We denote by $\pshfr$ the category of presheaves on $\ZZ \mathrm{F}_*(k)$, by $\mathcal Q\pshfr$ the category of quasi-stable presheaves on $\ZZ \mathrm{F}_*(k)$ and by $\mathcal S\pshfr$ the category of stable presheaves.
\end{df}

Now, the functor $\alpha^\prime:\ZZ \mathrm{F}_*(k)\to \smc$ induces a functor $\psht\to \pshfr$. By Example \ref{ex:stable}, this functor induces a functor
\[
(\alpha^\prime)^*:\psht\to \mathcal S\pshfr.
\]
Recall next that a presheaf $F$ on $\sm$ is $\AA^1$-invariant if the map $F(X)\to F(X\times \AA^1)$ induced by the projection $X\times \AA^1\to X$ is an isomorphism for any smooth scheme $X$. A Nisnevich sheaf of abelian groups $F$ is strictly $\AA^1$-invariant if the homomorphisms $\H^i_{\nis}(X,F)\to \H^i_{\nis}(X\times \AA^1,F)$ induced by the projection are isomorphisms for $i\geq 0$.

We can now state the main theorem of \cite{Garkusha15}. 

\begin{thm}\label{thm:A1local_framedPSh}
For any $\AA^1$-invariant quasi-stable $\ZZ \mathrm{F}_*(k)$-presheaf of abelian groups $F$, the associated Nisnevich sheaf $F_{\nis}$ is $\AA^1$-invariant if the base field is infinite. Moreover, if the base field $k$ is infinite perfect, every $\AA^1$-invariant quasi-stable Nisnevich sheaf (of abelian groups) is strictly $\AA^1$-invariant and quasi-stable.
\end{thm}

\section{$\mathrm{MW}$-motivic complexes}\label{sec:MWmotives}

\subsection{Derived category}

For any abelian category $\mathcal A$, we denote by $\Comp(\mathcal A)$ the category of (possibly unbounded) complexes of objects of $\mathcal A$ and by $\K(\mathcal A)$ the category of complexes with morphisms up to homotopy. Finally, we denote by $\Der(\mathcal A)$ the derived category of $\Comp(\mathcal A)$. We refer to \cite[\S 10]{Weibel94} for all these notions.

\begin{num}
Recall from our notations that $t$ is now either the Nisnevich
 or the \'etale topology.

 As usual in motivic homotopy theory, our first task is to equip the
 category of complexes of $\mathrm{MW}$-$t$-sheaves with a good model structure.
 This is done using the method of \cite{CD1}, thanks to Lemma \ref{lm:corr_main}
 and the fact that $\sh$ is a Grothendieck abelian category
 (Proposition \ref{prop:exist_associated-W-t-sheaf}(2)).

Except for one subtlety in the case of the étale topology,
 our construction is analogous to that of sheaves with transfers.
 In particular, the proof of the main point is essentially
 an adaptation of \cite[5.1.26]{CD3}. In order to make a short and streamlined proof, we first recall a few
 facts from model category theory.
\end{num}

\begin{num}\label{num:model}
We will be using the adjunction of Grothendieck abelian categories:
$$
\tilde \gamma^*:\sh \leftrightarrows \sht:\tilde \gamma_*
$$
of Corollary \ref{cor:adjunctions_corr}. Recall from Lemma \ref{lm:compare_transfers} that the functor
 $\tilde \gamma_*$ is conservative and exact.

First, there exists the so-called injective model structure
 on $\Comp(\sh)$ and $\Comp(\sht)$ which is defined such that
 the cofibrations are monomorphisms (thus every object is cofibrant)
 and weak equivalences are quasi-isomorphisms (this is classical;
 see e.g. \cite[2.1]{CD1}). The fibrant objects for this model structure
 are called \emph{injectively fibrant}.

Second, there exists the $t$-descent model structure on the category
 $\Comp(\sh)$ (see \cite[Ex. 2.3]{CD1}) characterized by the following properties:
\begin{itemize}
\item the class of \emph{cofibrations} is given by the smallest class
 of morphisms of complexes closed under suspensions, pushouts,
 transfinite compositions and retracts generated by the inclusions
\begin{equation} \label{eq:model1}
R_t(X) \rightarrow C\big(R_t(X) \xrightarrow{Id} R_t(X)\big)[-1]
\end{equation}
for a smooth scheme $X$, where $R_t(X)$ is the free sheaf of $R$-modules on $X$.
\item weak equivalences are quasi-isomorphisms.
\end{itemize}
Our aim is to obtain the same kind of model structure
 on the category $\Comp(\sht)$ of complexes of $\mathrm{MW}$-$t$-sheaves.
 Let us recall from \cite{CD1} that one can describe nicely the fibrant
 objects for the $t$-descent model structure.
 This relies on the following definition 
 for a complex $K$ of $t$-sheaves:
\begin{itemize}
\item the complex $K$ is \emph{local} if for any smooth scheme $X$
 and any integer $n \in \ZZ$, the canonical map:
\begin{equation} \label{eq:model2}
\H^n\big(K(X)\big)=\Hom_{\K(\sh)}(R_t(X),K[n])
 \rightarrow \Hom_{\Der(\sh)}(R_t(X),K[n])
\end{equation}
is an isomorphism;
\item the complex $K$ is \emph{$t$-flasque} if for any smooth
 scheme $X$ and any $t$-hypercover $p:\cX \rightarrow X$,
 the induced map:
\begin{equation} \label{eq:model3}
\H^n\big(K(X)\big)=\Hom_{\K(\sh)}(R_t(X),K[n])
 \xrightarrow{p^*} \Hom_{\K(\sh)}(R_t(\cX),K[n])=\H^n\big(K(\cX)\big)
\end{equation}
is an isomorphism.

Our reference for $t$-hypercovers is \cite{DHI}. Recall in particular that
 $\cX$
 is a simplicial scheme whose terms are arbitrary direct sums of smooth schemes.
 Then the notation $R_t(\cX)$ stands for the complex associated
 with the simplicial $t$-sheaves obtained by applying
 the obvious extension of the functor $R_t$ to the category of
 direct sums of smooth schemes.
 Similarly, $K(\cX)$ is the total complex (with respect to products)
 of the obvious double complex.
\end{itemize}
Then, let us state for further reference the following theorem (\cite[Theorem 2.5]{CD1}).
\end{num}
\begin{thm}\label{thm:CD}
Let $K$ be a complex of $t$-sheaves on the smooth site. Then the following
 three properties on $K$ are equivalent:
\begin{enumerate}
\item[(i)] $K$ is fibrant for the $t$-descent model structure,
\item[(ii)] $K$ is local,
\item[(iii)] $K$ is $t$-flasque.
\end{enumerate}
\end{thm}
Under these equivalent conditions,
 we will say that $K$ is $t$-fibrant.\footnote{Moreover, fibrations
 for the $t$-descent model
 structure  are epimorphisms of complexes whose kernel is $t$-fibrant.}

\begin{num}\label{num:Wmodel}
Consider now the case of $\mathrm{MW}$-$t$-sheaves.
 We will define \emph{cofibrations} in $\Comp(\sht)$ as in the previous
 paragraph by replacing $R_t$ by $\tilde R_t$ in \eqref{eq:model1}, i.e. 
 the cofibrations are the morphisms in the smallest class
 of morphisms of complexes of $\mathrm{MW}$-$t$-sheaves closed under suspensions, pushouts,
 transfinite compositions and retracts generated by the inclusions
\begin{equation} 
\tilde R_t(X) \rightarrow C\big(\tilde R_t(X) \xrightarrow{Id} \tilde R_t(X)\big)[-1]
\end{equation}
for a smooth scheme $X$. In particular, note that bounded above complexes of $\mathrm{MW}$-$t$-sheaves whose components are direct sums of sheaves of the form $\tilde R_t(X)$ are cofibrant. This is easily seen by taking the push-out of (\ref{eq:model1}) along the morphism $\tilde R_t(X)\to 0$.

 Similarly, a complex $K$ in $\Comp(\sht)$ will be called
 \emph{local} (resp. \emph{$t$-flasque}) if it satisfies
 the definition in the preceding paragraph
 after replacing respectively $\sh$ and $R_t$
 by $\sht$ and $\tilde R_t$ in \eqref{eq:model2} (resp. \eqref{eq:model3}).

In order to show that cofibrations and quasi-isomorphisms define
 a model structure on $\Comp(\sht)$, we will have to prove
 the following result in analogy with the previous theorem.
\end{num}

\begin{thm}\label{thm:decent_Wsheaves}
Let $K$ be a complex of $\mathrm{MW}$-$t$-sheaves.
 Then the following conditions are equivalent:
\begin{enumerate}
\item[(i)] $K$ is local;
\item[(ii)] $K$ is $t$-flasque.
\end{enumerate}
\end{thm}

The proof is essentially an adaptation of the proof of
 \cite[5.1.26, 10.3.17]{CD3},
 except that the case of the \'etale topology needs a new argument. It will be completed as a corollary of two lemmas, the first of which is a reinforcement of Lemma \ref{lm:corr_main}.

\begin{lm}\label{lm:corr_main_strong}
Let $p:\cX \rightarrow X$ be a $t$-hypercover of a smooth scheme
 $X$. Then the induced map:
$$
p_*:\tilde R_t(\cX) \rightarrow \tilde R_t(X)
$$
is a quasi-isomorphism of complexes of $\mathrm{MW}$-$t$-sheaves.
\end{lm}
\begin{proof}
In fact, we have to prove that the complex $\tilde R_t(\cX)$
 is acyclic in positive degree and that $p_*$
 induces an isomorphism $\H_0(\tilde R_t(\cX))=\tilde R_t(X)$.\footnote{Note that
 the second fact follows from Lemma \ref{lm:corr_main}
 and the definition of $t$-hypercovers, but our proof works more directly.}
 In particular, as these assertions only concerns the $n$-th homology
 sheaf of $\tilde R_t(\cX)$, we can always assume that
 $\cX \simeq \mathrm{cosk}_n(\cX)$ for a
 large enough integer $n$ (because these two simplicial objects
 have the same $(n-1)$-skeleton). In other words, we can assume
 that $\cX$ is a bounded $t$-hypercover in the terminology
 of \cite[Def. 4.10]{DHI}.

As a consequence of the existence of the injective model
 structure, the category $\Der(\sht)$ is naturally enriched over the derived
 category of $R$-modules. Let us denote by $\derR \Hom^\bullet$ the
 corresponding Hom-object.
 We have only to prove that for any complex $K$
 of $\mathrm{MW}$-$t$-sheaves, the natural map:
$$
p^*:\derR \Hom^\bullet(\tilde R_t(X),K)
 \rightarrow \derR \Hom^\bullet(\tilde R_t(\cX),K)
$$
is an isomorphism in the derived category of $R$-modules.
 Because there exists an injectively fibrant resolution of any complex $K$,
 and $\derR \Hom^\bullet$ preserves quasi-isomorphisms,
 it is enough to consider the case of an injectively fibrant
 complex $K$ of $\mathrm{MW}$-$t$-sheaves.

In this case, $\derR \Hom^\bullet(-,K)=\Hom^\bullet(-,K)$
 (as any complex is cofibrant for the injective model structure)
 and we are reduced to prove that the following complex of presheaves on
 the smooth site:
$$
X \mapsto \Hom^\bullet(\tilde R_t(X),K)
$$
satisfies $t$-descent with respect to bounded $t$-hypercovers
 \emph{i.e.} sends bounded $t$-hypercovers $\cX/X$ to quasi-isomorphisms of complexes
 of $R$-modules. But Lemma \ref{lm:corr_main} (and the fact that $K$ is injectively fibrant)
 tells us that this is the case when $\cX$ is the $t$-hypercover associated
 with a $t$-cover. So we conclude using \cite[A.6]{DHI}.
\end{proof}

The second lemma for the proof of Theorem \ref{thm:decent_Wsheaves}
 is based on the previous one.
 
\begin{lm}\label{lm:decent_Wsheaves2}
Let us denote by $\Comp$, $\K$, $\Der$
 (respectively by  $\tilde \Comp$, $\tilde \K$, $\tilde \Der$)
 the category of complexes, complexes up to homotopy and derived
 category of the category $\sh$ (respectively $\sht$).

 Given a simplicial scheme $\cX$ whose components are (possibly infinite) coproducts of smooth $k$-schemes
 and a complex $K$ of $\mathrm{MW}$-$t$-sheaves, 
 we consider the isomorphism of $R$-modules obtained
 from the adjunction $(\tilde \gamma^*,\tilde \gamma_*)$:
$$
\epsilon_{\cX,K}:\Hom_{\tilde \Comp}(\tilde R_t(\cX),K)
\rightarrow \Hom_{\Comp}(R_t(\cX),\tilde \gamma_*(K)).
$$
Then there exist unique isomorphisms $\epsilon'_{\cX,K}$
 and $\epsilon''_{\cX,K}$ of $R$-modules making the following diagram
 commutative:
$$
\xymatrix@R=18pt@C=28pt{
\Hom_{\tilde \Comp}(\tilde R_t(\cX),K)\ar^-{\epsilon_{\cX,K}}[r]\ar[d]
 & \Hom_{\Comp}(R_t(\cX),\tilde \gamma_*(K))\ar[d] \\
\Hom_{\tilde \K}(\tilde R_t(\cX),K)\ar^-{\epsilon'_{\cX,K}}[r]\ar_{\pi_{\mathcal X,K}}[d]
 & \Hom_{\K}(R_t(\cX),\tilde \gamma_*(K))\ar^{\pi'_{\mathcal X,K}}[d] \\
\Hom_{\tilde \Der}(\tilde R_t(\cX),K)\ar^-{\epsilon''_{\cX,K}}[r]
 & \Hom_{\Der}(R_t(\cX),\tilde \gamma_*(K))
}
$$
where the vertical morphisms are the natural localization maps.
\end{lm}

\begin{proof}
The existence and unicity of $\epsilon'_{\cX,K}$ simply follows
 from the fact $\tilde \gamma^*$ and $\tilde \gamma_*$ are additive functors,
 so in particular $\epsilon_{\cX,K}$ is compatible with chain homotopy
 equivalences.
 
For the case of $\epsilon_{\cX,K}^{''}$, we assume that 
 the complex $K$ is injectively fibrant. In this case,
 the map $\pi_{\mathcal X,K}$ is an isomorphism.
 This already implies the existence and unicity of the map
 $\epsilon''_{\cX,K}$. Besides, according to the previous lemma and the fact that the map
 $\pi_{\mathcal X,K}$ is an isomorphism natural in $\mathcal X$,
 we obtain that $K$ is $t$-flasque
  (in the sense of Paragraph \ref{num:Wmodel}).
 Because $\epsilon'_{\cX,K}$ is an isomorphism natural in $\mathcal X$,
 we deduce that $\tilde \gamma_*(K)$ is $t$-flasque.
 In view of Theorem \ref{thm:CD}, it is $t$-fibrant. 
 As $R_t(\cX)$ is cofibrant for the $t$-descent model
 structure on $\Comp$, we deduce that $\pi'_{\mathcal X,K}$ is an isomorphism.
 Therefore, $\epsilon''_{\cX,K}$ is an isomorphism.

The case of a general complex $K$ now follows from the existence of an injectively
 fibrant resolution $K \rightarrow K'$ of any complex of $\mathrm{MW}$-$t$-sheaves $K$.
\end{proof}

\begin{proof}[proof of Theorem \ref{thm:decent_Wsheaves}]
 The previous lemma shows that the following conditions on
 a complex $K$ of $\mathrm{MW}$-$t$-sheaves are equivalent:
\begin {itemize}
\item $K$ is local (resp. $t$-flasque) in $\Comp(\sht)$;
\item $\tilde \gamma_*(K)$ is local (resp. $t$-flasque) in $\Comp(\sh)$.
\end {itemize}
Then Theorem \ref{thm:decent_Wsheaves} follows from Theorem \ref{thm:CD}.
\end{proof}

Here is an important corollary (analogous to \cite[chap. 5, 3.1.8]{FSV}) which is simply a restatement of Lemma \ref{lm:decent_Wsheaves2}. 

\begin{cor}\label{cor:compare_Hom&cohomology}
Let $K$ be a complex of $\mathrm{MW}$-$t$-sheaves and $X$ be a smooth scheme.
 Then for any integer $n \in \ZZ$, there exists a canonical isomorphism,
 functorial in $X$ and $K$:
$$
\Hom_{\Der(\sht)}(\tilde R_t(X),K[n])=\mathbb{H}^n_t(X,K)
$$
where the right hand side stands for the $t$-hypercohomology
of $X$ with coefficients in the complex $\tilde \gamma_*(K)$
 (obtained after forgetting $\mathrm{MW}$-transfers).
\end{cor}

Recall that the category $\Comp(\sht)$ is symmetric monoidal,
 with tensor product induced as usual from the tensor product
 on $\sht$ (see Paragraph \ref{num:sht_monoidal}).

\begin{cor}\label{cor:model_Der}
The category $\Comp(\sht)$
 has a proper cellular model structure
  (see \cite[12.1.1 and 13.1.1]{Hirschhorn03})
 with quasi-isomorphisms as weak equivalences and
 cofibrations as defined in Paragraph \ref{num:Wmodel}.
Moreover, the fibrations for this model structure
 are epimorphisms of complexes whose kernel
 are $t$-flasque (or equivalently local) complexes of $\mathrm{MW}$-$t$-sheaves.
Finally, this is a symmetric monoidal model structure;
 in other words, the tensor products (resp. internal Hom functor)
 admits a total left (resp. right) derived functor.
\end{cor}

\begin{proof}
Each claim is a consequence of \cite[2.5, 5.5 and 3.2]{CD1},
 applied to the Grothendieck abelian category $\sht$ with respect
 to the descent structure $(\mathcal G,\mathcal H)$ (see \cite[Def. 2.2]{CD1}
 for the notion of descent structure) defined as follows:
\begin{itemize}
\item $\mathcal G$ is the class of $\mathrm{MW}$-$t$-sheaves of the form $\tilde R_t(X)$
 for smooth scheme $X$;
\item $\mathcal H$ is the (small) family of complexes which are
 cones of morphisms $p_*:\tilde R_t(\cX) \rightarrow \tilde R_t(X)$
 for a $t$-hypercover $p$.
\end{itemize}
Indeed, $\mathcal G$ generates the category $\sht$
 (see after Definition \ref{df:gen_sht}) and the condition
 to be a descent structure is given by Theorem \ref{thm:decent_Wsheaves}.

In the end, we can apply \cite[3.2]{CD1} to derive the
 tensor product as the tensor structure is weakly flat (in the sense of \cite[\S 3.1]{CD1})
 due to the preceding definition and formula \eqref{eq:sht_monoidal}.
\end{proof}

\begin{rem}
We can follow the procedure of \cite[\S 8]{Mazza06} to compute the tensor product of two bounded above complexes of $\mathrm{MW}$-$t$-sheaves. This follows from \cite[Proposition 3.2]{CD1} and the fact that bounded above complexes of $\mathrm{MW}$-$t$-sheaves whose components are direct sums of representable sheaves are cofibrant.
\end{rem}

\begin{df}
The model structure on $\Comp(\sht)$ of the above corollary is called the \emph{$t$-descent model structure}.
\end{df}

In particular, the category $\Der(\sht)$
 is a triangulated symmetric closed monoidal category.

\begin{num}
We also deduce from the $t$-descent model structure
 that the vertical adjunctions of Corollary \ref{cor:adjunctions_corr}
 induce Quillen adjunctions with respect to the $t$-descent model
 structure on each category involved and so admit derived functors
 as follows:
\begin{equation}\label{eq:chg_top&tr_Der}
\begin{split}
\xymatrix@C=30pt@R=24pt{
\Der(\shx{})\ar@<+2pt>^{\derL \tilde \gamma^*}[r]\ar@<+2pt>^{a}[d]
 & \Der(\shtx{})\ar@<+2pt>^{\derL \pi^*}[r]\ar@<+2pt>^{\tilde a}[d]
     \ar@<+2pt>^{\tilde \gamma_{*}}[l]
 & \Der(\shtVx{})\ar@<+2pt>^{a^{\mathrm{tr}}}[d]
     \ar@<+2pt>^{\pi_{*}}[l] \\
\Der(\shx{\et})\ar@<+2pt>^{\derL \tilde \gamma^*_\et}[r]
    \ar@<+2pt>^{\derR \mathcal O}[u]
 & \Der(\shtx{\et})\ar@<+2pt>^{\derL \pi^*_\et}[r]
	  \ar@<+2pt>^{\derR \mathcal O}[u]
		\ar@<+2pt>^{\tilde \gamma_{\et*}}[l]
 & \Der(\shtVx{\et})
    \ar@<+2pt>^{\derR \mathcal O}[u]
		\ar@<+2pt>^{\pi_{\et*}}[l]
}
\end{split}
\end{equation}
where we have not indicated the topology in the notations
 when it is the Nisnevich topology, denoted by $(a,\mathcal O)$
 for the adjoint pair associated \'etale sheaf and forgetful functor
 and similarly for $\mathrm{MW}$-transfers and transfers.
When the functors are exact, they are trivially derived and so we
 have used the same notation than for their counterpart for
 sheaves.

Note that by definition, the left adjoints in this diagram
 are all monoidal functors and sends the object represented
 by a smooth scheme $X$ (say in degree $0$) to the analogous
 object
\begin{equation}\label{eq:derived&rep}
\derL\tilde \gamma^*\big(R_t(X)\big)=\tilde R_t(X),
 \derL\pi^*\big(\tilde R_t(X)\big)=R^{\mathrm{tr}}(X).
\end{equation}
\end{num}

%
%

\subsection{The $\AA^1$-derived category}

We will now adapt the usual $\AA^1$-homotopy machinery to our context.

\begin{df}
We define the category $\DMtex t$ of $\mathrm{MW}$-motivic complexes
 for the topology $t$
 as the localization of the triangulated category
 $\Der(\sht)$ with respect to the localizing triangulated
 subcategory\footnote{Recall that according to Neeman \cite[3.2.6]{Nee},
 localizing means stable by coproducts.}
 $\mathcal T_{\AA^1}$
 generated by complexes of the form:
$$
\cdots 0 \rightarrow \tilde R_t(\AA^1_X) \xrightarrow{p_*} \tilde R_t(X) \rightarrow 0 \cdots
$$
where $p$ is the projection of the affine line relative to
 an arbitrary smooth $k$-scheme $X$.
 As usual, we define the \emph{$\mathrm{MW}$-motive $\mot(X)$}
 associated to a smooth scheme $X$ as the complex concentrated in degree
 $0$ and equal to the representable $\mathrm{MW}$-$t$-sheaf
 $\tilde{R}_t(X)$.  Respecting our previous conventions,
 we mean the Nisnevich topology when the topology is not indicated.
\end{df}

According to this definition, it is formal that the localizing
 triangulated subcategory $\mathcal T_{\AA^1}$ is
 stable under the derived tensor product of $\Der(\sht)$
 (cf. Corollary \ref{cor:model_Der}). In particular, it induces
 a triangulated monoidal structure on $\DMtex t$.

\begin{num}\label{num:A1-local}
As usual, we can apply the classical techniques of localization
 to our triangulated categories
 and also to our underlying model structure.
 So a complex of $\mathrm{MW}$-$t$-sheaf $E$ is called \emph{$\AA^1$-local}
 if for any smooth scheme $X$ and any integer $i \in \ZZ$,
 the induced map
$$
\Hom_{\Der(\sht)}(\tilde R_t(X),E[i])
 \rightarrow \Hom_{\Der(\sht)}(\tilde R_t(\AA^1_X),E[i])
$$
is an isomorphism. In view of Corollary \ref{cor:compare_Hom&cohomology},
 it amounts to ask that the $t$-cohomology of $\tilde \gamma_*(E)$ is
 $\AA^1$-invariant, or in equivalent words, that $E$ is strictly
 $\AA^1$-local.

Applying Neeman's localization theorem (see
 \cite[9.1.19]{Nee}),\footnote{Indeed recall the derived
 category of $\sht$ is a well generated triangulated category
 according to \cite[Th. 0.2]{Nee2}.} the category
 $\DMtex t$ can be viewed as the full subcategory of $\Der(\sht)$
 whose objects are the $\AA^1$-local complexes. Equivalently,
 the canonical functor $\Der(\sht) \rightarrow \DMtex t$
 admits a fully faithful right adjoint whose essential image
 consists in  $\AA^1$-local complexes. In particular,
 one deduces formally the existence of an $\AA^1$-localization
 functor $L_{\AA^1}:\Der(\sht) \rightarrow \Der(\sht)$.

Besides, we get the following proposition by applying the 
 general left Bousfield localization procedure for proper cellular
 model categories (see \cite[4.1.1]{Hirschhorn03}). 
 We say that a morphism $\phi$ of $\Der(\sht)$
 is an \emph{weak $\AA^1$-equivalence} if for any $\AA^1$-local object $E$,
 the induced map $\Hom(\phi,E)$ is an isomorphism.
\end{num}
\begin{prop}\label{prop:model_DMt}
The category $\Comp(\sht)$
 has a symmetric monoidal model structure
 with weak $\AA^1$-equivalences as weak equivalences and
 cofibrations as defined in Paragraph \ref{num:Wmodel}.
 This model structure is proper and cellular. Moreover, the fibrations for this model structure
 are epimorphisms of complexes whose kernel
 are $t$-flasque and
 $\AA^1$-local complexes.
\end{prop}

The resulting model structure on $\Comp(\sht)$
 will be called the \emph{$\AA^1$-local model structure}.
 The proof of the proposition follows formally from
 Corollary \ref{cor:model_Der} by the usual localization
 procedure of model categories, see \cite[\textsection 3]{CD1}
 for details. Note that the tensor product of two bounded above complexes can be computed as in the derived category. 

\begin{num}
As a consequence of the above discussion, the category $\DMtex t$ is a triangulated
 symmetric monoidal closed category. Besides, it is clear
 that the functors of Corollary \eqref{cor:adjunctions_corr}
 induce Quillen adjunctions for the $\AA^1$-local model
 structures. Equivalently, Diagram \eqref{eq:chg_top&tr_Der}
 is compatible with $\AA^1$-localization and induces
 adjunctions of triangulated categories:
\begin{equation}\label{eq:chg_top&tr_DMeff}
\begin{split}
\xymatrix@C=30pt@R=24pt{
\DAe\ar@<+2pt>^{\derL \tilde \gamma^*}[r]\ar@<+2pt>^{a}[d]
 & \DMte\ar@<+2pt>^{\derL \pi^*}[r]\ar@<+2pt>^{\tilde a}[d]
     \ar@<+2pt>^{\tilde \gamma_{*}}[l]
 & \DMe\ar@<+2pt>^{a^{tr}}[d]
     \ar@<+2pt>^{\pi_{*}}[l] \\
\DAex{\et}\ar@<+2pt>^{\derL \tilde \gamma^*_\et}[r]
    \ar@<+2pt>^{\derR \mathcal O}[u]
 & \DMtex{\et}\ar@<+2pt>^{\derL \pi^*_\et}[r]
	  \ar@<+2pt>^{\derR \mathcal O}[u]
		\ar@<+2pt>^{\tilde \gamma_{\et*}}[l]
 & \DMex{\et}.
    \ar@<+2pt>^{\derR \mathcal O}[u]
		\ar@<+2pt>^{\pi_{\et*}}[l]
}
\end{split}
\end{equation}
In this diagram, the left adjoints are all monoidal and send
 the different variant of motives represented by a smooth scheme $X$
 to the analogous motive. In particular,
 $$\derL \pi^*\mot(X)=\motV(X).$$
Also, the functors $\tilde \gamma_{t*}$ and $\pi_{t*}$
 for $t=\nis, \et$ (or following our conventions, $t=\varnothing, \et$)
 are conservative. Note moreover that their analogues in diagram
 \eqref{eq:chg_top&tr_Der} preserve $\AA^1$-local objects
 and so commute with the $\AA^1$-localization functor.
 Therefore one deduces from Morel's $\AA^1$-connectivity
 theorem \cite{Morel05b} the following result.
\end{num}
\begin{thm}
Assume $k$ is a perfect field. Let $E$ be a complex of $\mathrm{MW}$-sheaves concentrated in positive degrees.
 Then the complex $L_{\AA^1} E$ is concentrated in positive degrees.
\end{thm}
Indeed, to check this, one needs only to apply the functor 
 $\tilde \gamma_*$ as it is conservative and so we are reduced to Morel's
 theorem \cite[6.1.8]{Morel05b}.

\begin{cor}\label{cor:pre-df:htp_tstruct_eff}
Under the assumption of the previous theorem, 
 the triangulated category $\DMte$ admits a unique $t$-structure
 such that the functor 
 $\tilde \gamma_*:\DMte \rightarrow \Der\big(\shtx{}\big)$
 is $t$-exact.
\end{cor}
Note that the truncation functor on $\DMte$ is obtained by applying
 to an $\AA^1$-local complex the usual truncation functor of
 $\Der(\shtx{})$ and then the $\AA^1$-localization functor.

\begin{df}\label{df:htp_tstruct_eff}
If $k$ is a perfect field, 
 the $t$-structure on $\DMte$ obtained in the previous corollary
 will be called the \emph{homotopy $t$-structure}.
\end{df}


\begin{rem}
Of course, the triangulated categories $\DMe$ and $\DAe$
 are also equipped with $t$-structures, called in each 
 cases homotopy $t$-structures
 --- in the first case, it is due to Voevodsky and in the second
 to Morel.

It is clear from the definitions that the functors
 $\tilde \gamma_*$ and $\pi_*$ in Diagram \eqref{eq:chg_top&tr_DMeff}
 are $t$-exact.
\end{rem}

As in the case of Nisnevich sheaves with transfers,
 we can describe nicely $\AA^1$-local objects and
 the $\AA^1$-localization functor due to the following theorem.
 
\begin{thm}\label{thm:Wsh&A1}
Assume $k$ is an infinite perfect field. Let $F$ be an $\AA^1$-invariant $\mathrm{MW}$-presheaf. Then, the associated $\mathrm{MW}$-sheaf $\tilde a(F)$ is
 strictly $\AA^1$-invariant. Moreover, the Zariski sheaf associated with $F$
 coincides with $\tilde a(F)$ and the natural map
$$
\H^i_\zar(X,\tilde a(F)) \rightarrow
\H^i_\nis(X,\tilde a(F))
$$
is an isomorphism for any smooth scheme $X$.
\end{thm}

\begin{proof}
Recall that we have a functor $(\alpha^\prime)^*:\psht\to \mathcal S\pshfr$. In view of Theorem \ref{thm:A1local_framedPSh}, the Nisnevich sheaf associated to the presheaf $(\alpha^\prime)^*(F)$ is strictly $\AA^1$-invariant and quasi-stable. It follows that $\tilde a(F)$ is strictly $\AA^1$-invariant by Corollary \ref{cor:adjunctions_corr}. 
Now, a strictly $\AA^1$-invariant sheaf admits a Rost-Schmid complex in the sense of \cite[\S 5]{Morel08} and the result follows from \cite[Corollary 5.43]{Morel08}. 
\end{proof}

\begin{rem}
It would be good to have a proof intrinsic to $\mathrm{MW}$-motives of the above theorem. This will be worked out in H\r{a}kon Andreas Kolderup's thesis (\cite{Kolderup17}).
\end{rem}

As in the case of Voeovdsky's theory of motivic complexes,
 this theorem has several important corollaries.
 
\begin{cor}\label{cor:A1-local_complexes}
Assume $k$ is an infinite perfect field. Then, a complex $E$ of $\mathrm{MW}$-sheaves is $\AA^1$-local
 if and only if its homology (Nisnevich) sheaves 
 are $\AA^1$-invariant.
 The heart of the homotopy $t$-structure (Def. \ref{df:htp_tstruct_eff})
 on $\DMte$ consists of $\AA^1$-invariant $\mathrm{MW}$-sheaves.
\end{cor}
The proof is classical. We recall it for the comfort of the reader.

\begin{proof}
Let $K$ be an $\AA^1$-local complex of $\mathrm{MW}$-sheaves over $k$.
 Let us show that its homology sheaves are $\AA^1$-invariant.
 According to the existence of the $\AA^1$-local model structure
 on $\Comp(\sh)$ (Proposition \ref{prop:model_DMt}),
 there exist a Nisnevich fibrant and $\AA^1$-local complex
 $K'$ and a weak $\AA^1$-equivalence:
$$
K \xrightarrow \phi K'.
$$
 As $K$ and $K'$ are $\AA^1$-local, the map $\phi$ is a
 quasi-isomorphism. In particular, we can replace $K$ by $K'$.
 In other words, we can assume $K$ is Nisnevich fibrant thus local
 (Theorem \ref{thm:CD}). 
 Then we get:
$$
\H^n\big(K(X)\big) \simeq
 \Hom_{\Der(\sht)}\big(\repR X,K[n]\big)
$$
according to the definition of local in Paragraph \ref{num:Wmodel}.
 This implies in particular that the cohomology presheaves of $K$ are
 $\AA^1$-invariant. We conclude using Theorem \ref{thm:Wsh&A1}.

Assume conversely that $K$ is a complex of $\mathrm{MW}$-sheaves
 whose homology sheaves are $\AA^1$-invariant. Let us show that
 $K$ is $\AA^1$-local. According to Corollary
 \ref{cor:compare_Hom&cohomology}, we need only to show its Nisnevich
 hypercohomology is $\AA^1$-invariant. Then we apply the Nisnevich
 hypercohomology spectral sequence for any smooth scheme $X$:
$$
E_2^{p,q}=\H^p_\nis(X,\H^q_\nis(K)) \Rightarrow \mathbb{H}^n_\nis(X,K).
$$
As the cohomological dimension of the Nisnevich topology
 is bounded by the dimension of $X$, the $E_2$-term is concentrated
 in degree $p \in [0,\dim(X)]$ and the spectral sequence converges (\cite[Theorem 0.3]{Suslin00}).
 It is moreover functorial in $X$. Therefore it is enough
 to show the map induced by the projection
$$
\H^p_\nis(X,\H^q_\nis(K)) \rightarrow \H^p_\nis(\AA^1_X,\H^q_\nis(K))
$$
is an isomorphism to conclude. By assumption the sheaf
 $\H^q_\nis(K)$ is $\AA^1$-invariant so Theorem \ref{thm:Wsh&A1} applies
 again.

As the functor $\tilde \gamma_*$ is $t$-exact by
 Corollary \ref{cor:pre-df:htp_tstruct_eff}, the conclusion 
 about the heart of $\DMte$ follows.
\end{proof}

\begin{cor}\label{cor:cohomcomplex}
Let $K$ be an $\AA^1$-local complex of $\mathrm{MW}$-sheaves. Then we have 
 \[
 \mathbb{H}^i_{\zar}(X,K)= \mathbb{H}^i_{\nis}(X,K)
 \]
for any smooth scheme $X$ and any $i\in \ZZ$.
\end{cor}
\begin{proof}
The proof uses the same principle as in the previous result.
 Let us first consider the change of topology adjunction:
$$
\alpha^*:\shx \zar \leftrightarrows \shx \nis:\alpha_*
$$
where $\alpha^*$ is obtained using the functor ``associated
 Nisnevich sheaf functor"
 and $\alpha_*$ is just the forgetful functor.
 This adjunction can be derived (using for example the injective
 model structures) and induces:
$$
\alpha^*:\Der(\shx \zar) \leftrightarrows \Der(\shx \nis):\derR \alpha_*
$$
--- note indeed that $\alpha^*$ is exact.
 Coming back to the statement of the corollary, we have to show
 that the adjunction map:
\begin{equation}\label{eq:compar_zar/nis}
 \tilde \gamma_*(K) \rightarrow \derR \alpha_* \alpha^*( \tilde\gamma_*(K))
\end{equation}
is a quasi-isomorphism. Let us denote abusively by $K$ the
 sheaf $\tilde \gamma_*(K)$. Note that this will be harmless as
 $\tilde \gamma_*(\H^q_\nis(K))=\H^q_\nis(\tilde \gamma_*K)$.
 With this abuse of notation, one has canonical identifications:
\begin{align*}
\Hom_{\Der(\shx \nis)}(\rep X,\derR \alpha_* \alpha^*(K))
 &=\H^p_\zar(X,K) \\
\Hom_{\Der(\shx \nis)}(\rep X,\derR \alpha_* \alpha^*(\H^q_\nis(K))
 &=\H^p_\zar(X,\H^q_\nis(K))
\end{align*}
 Using now the tower of truncation
 of $K$ for the standard $t$-structure on $\Der(\sht)$
 --- or equivalently $\Der(\sh)$ ---
 and the preceding identifications, one gets a spectral sequence:
$$
^{\zar}E_2^{p,q}=\H^p_\zar(X,\H^q_\nis(K))
 \Rightarrow \mathbb{H}^{p+q}_\zar(X,K)
$$
and the morphism \eqref{eq:compar_zar/nis} induces a morphism
 of spectral sequence:
\begin{align*}
E_2^{p,q}=\H^p_\nis(X,\H^q_\nis(K))
 \longrightarrow &^{\zar}E_2^{p,q}=\H^p_\zar(X,\H^q_\nis(K)) \\
 & \Rightarrow \mathbb{H}^{p+q}_\nis(X,K)
 \longrightarrow \mathbb{H}^{p+q}_\zar(X,K).
\end{align*}
The two spectral sequences converge (as the Zariski and cohomological
 dimension of $X$ are bounded). According to Theorem \ref{thm:Wsh&A1},
 the map on the $E_2$-term is an isomorphism so the map
 on the abutment must be an isomorphism and this concludes.
\end{proof}

\begin{num}
Following Voevodsky, given a complex $E$ of $\mathrm{MW}$-sheaves, we define its
 associated Suslin complex as the following complex of
 sheaves:\footnote{Explicitly, this complex associates
 to a smooth scheme $X$ the total complex (for coproducts)
 of the bicomplex $E(\Delta^* \times X)$.}
$$
\sus E:=\uHom(\tilde R_t(\Delta^*),E)
$$
where $\Delta^*$ is the standard cosimplicial scheme.
\end{num}

\begin{cor}\label{cor:LA1}
Assume $k$ is an infinite perfect field.

Then for any complex $E$ of $\mathrm{MW}$-sheaves, there
 exists a canonical quasi-isomorphism:
$$
L_{\AA^1}(E) \simeq \sus E.
$$
\end{cor}

\begin{proof}
Indeed, according to Corollary \ref{cor:A1-local_complexes}, it is clear
 that $\sus E$ is $\AA^1$-local. Thus the canonical map
$$
c:E \rightarrow \sus E
$$
induces a morphism of complexes:
$$
L_{\AA^1}(c):L_{\AA^1}(E) \rightarrow L_{\AA^1}(\sus E) = \sus E.
$$
As $\Delta^n \simeq \AA^n_k$, one checks easily that
 the map $c$ is an $\AA^1$-weak equivalence. Therefore
 the map $L_{\AA^1}(c)$ is an $\AA^1$-weak equivalence
 of $\AA^1$-local complexes, thus a quasi-isomorphism.
\end{proof}

\begin{num}
As usual, one defines the Tate object in $\DMte$ by
 the formula
 \[
\tilde R(1):=\mot(\PP^1_k)/\mot(\{\infty\})[-2]
 \simeq \mot(\AA^1_k)/\mot(\AA^1_k-\{0\})[-2]
 \simeq \mot(\GG)/\mot(\{1\})[-1].
 \]
 Then one defines the effective $\mathrm{MW}$-motivic cohomology of a smooth
 scheme $X$ in degree $(n,i) \in \ZZ \times \NN$:\footnote{Negative
 twists will be introduced in the next section.} as
$$
\H^{n,i}_{\mathrm{MW}}(X,R)=\Hom_{\DMte}(\tilde R(X),\tilde R(i)[n])
$$
where $\tilde R(i)=\tilde R(1)^{\otimes i}$.
\end{num}

\begin{cor}\label{cor:W-motivic&generalized}
Assume $k$ is an infinite perfect field.
The effective $\mathrm{MW}$-motivic cohomology defined above coincides with 
 the generalized motivic cohomology groups
 defined (for $R=\ZZ$) in \cite[5.31]{Calmes14b}.
\end{cor}

\begin{proof}
By Corollary \ref{cor:LA1}, the Suslin complex of $R(i)$ is $\AA^1$-local. It follows then from Corollary \ref{cor:cohomcomplex} that its Nisnevich hypercohomology and its Zariski hypercohomology coincide.  We conclude using \cite[Corollary 4.0.4]{Fasel16}.
\end{proof}

We now spend a few lines in order to compare ordinary motivic cohomology with the version of Definition \ref{def:generalizedMW}, following \cite[Definition 6.7]{Calmes14b}. In this part, we suppose that $R$ is flat over $\ZZ$. If $X$ is a smooth scheme, recall from \cite[Definition 5.15]{Calmes14b} that the presheaf with $\mathrm{MW}$-transfers $\Icorr(X)$ defined by 
\[
\Icorr(X)(Y)=\ilim_T\H^d_T\!\big(X \times Y,\mathrm{I}^{d+1},\omega_{Y}\big)
\]
fits in an exact sequence
\[
0\to \Icorr(X)\to \corr(X)\to \ZZ^{\mathrm{tr}}(X)\to 0
\]
As $\corr(X)$ is a sheaf in the Zariski topology, it follows that $\Icorr(X)$ is also such a sheaf. We can also consider the Zariski sheaf $\Icorr_R(X)$ defined by 
\[
\Icorr_R(X)(Y)=\Icorr(X)(Y)\otimes R.
\]

\begin{df}
We denote by $\IrepR X$ the $t$-sheaf associated to the presheaf $\Icorr_R(X)$.
\end{df}

In view of Proposition \ref{prop:exist_associated-W-t-sheaf}, $\Icorr_R(X)$ has $\mathrm{MW}$-transfers. Moreover, sheafification being exact and $R$ being flat, we have an exact sequence
\[
0\to \IrepR X\to \repR X\to R^{\mathrm{tr}}(X)\to 0
\]
of $\mathrm{MW}$-$t$-sheaves (note the slight abuse of notation when we write $R^{\mathrm{tr}}(X)$ in place of $\pi_*^tR^{\mathrm{tr}}(X))$. We deduce from \cite[Lemma 2.13]{Mazza06} an exact sequence 
\[
0\to \mathrm{I}\tilde{R}_t\{q\}\to \tilde{R}_t\{q\}\to R^{\mathrm{tr}}_t\{q\}\to 0
\]
for any $q\in\NN$, where $\mathrm{I}\tilde{R}_t\{m\}=\IrepR {\GG^{\wedge m}}$.

\begin{df}
For any $q\in \NN$, we set $\mathrm{I}\tilde{R}_t(q)=\mathrm{I}\tilde{R}_t\{q\}[-q]$ in $\DMte$ and 
\[
\H_{\mathrm{I}}^{p,q}(X,R)=\Hom_{\DMte}(\tilde M(X),\mathrm{I}\tilde{R}_t(q)[p])
\]
for any smooth scheme $X$.
\end{df}

As Zariski cohomology and Nisnevich cohomology coincide by Corollary \ref{cor:cohomcomplex}, these groups coincide with the ones defined in \cite[Definition 6.7]{Calmes14b} (when $R=\ZZ$). In particular, we have a long exact sequence 
\[
\ldots\to \H_{\mathrm{I}}^{p,q}(X,R)\to \H_{\mathrm{MW}}^{p,q}(X,R)\to \H^{p,q}(X,R) \to \H_{\mathrm{I}}^{p+1,q}(X,R)\to \ldots
\]
for any smooth scheme $X$ and any $q\in \NN$.

\num To end this section,
 we now discuss the effective geometric $\mathrm{MW}$-motives,
 which are built as in the classical case.

\begin{df}
One defines the category $\DMtegm$ of \emph{geometric effective motives}
 over the field $k$ as the pseudo-abelianization
 of the Verdier localization of
 the homotopy category $\mathrm K^b(\smc)$ associated
 to the additive category $\smc$ with respect to the thick
 triangulated subcategory containing complexes of the form:
 
\begin{enumerate}
\item $\hdots \rightarrow [W] \xrightarrow{k_*-g_*} [V] \oplus [U]
 \xrightarrow{f_*+j_*} [X]
 \rightarrow \hdots$
 for an elementary Nisnevich distinguished square of smooth schemes:
$$
\xymatrix@=10pt{
W\ar_-g[d]\ar^-k[r] & V\ar^-f[d] \\
U\ar_-j[r] & X;
}
$$
\item $\hdots  \rightarrow [\AA^1_X] \xrightarrow{p_*} [X]  \rightarrow \hdots$
 where $p$ is the canonical projection and $X$ is a smooth scheme.
\end{enumerate}
\end{df}

It is clear that the natural map $\mathrm K^b(\smc) \rightarrow \Der(\shtZZ)$
 induces a canonical functor:
$$
\iota:\DMtegm \rightarrow \DMteZ.
$$
As a consequence of \cite[Theorem 6.2]{CD1} (see also Example 6.3 of \emph{op. cit.})
 and Theorem \ref{thm:Wsh&A1}, we get the following result.
\begin{prop}\label{prop:gm_objects}
The functor $\iota$ is fully faithful and its essential
 image coincides with each one of the following subcategories:
\begin{itemize}
\item the subcategory of compact objects of $\DMteZ$;
\item the smallest thick triangulated subcategory of $\DMteZ$
 which contains the motives $\mot(X)$ for a smooth scheme $X$.
\end{itemize}
Moreover, when $k$ is an infinite perfect field,
 one can reduce in point (1) in the definition of $\DMtegm$ to consider
 those complexes associated to a Zariski open cover $U \cup V$ of 
 a smooth scheme $X$.
\end{prop}

\begin{rem}\label{rem:compact_gen_DMte}
Note that the previous proposition states in particular that the objects
 of the form $\mot(X)$ for a smooth scheme $X$ are compact in $\DMte$.
 Therefore, they form a family of compact generators of $\DMte$ in the
 sense that $\DMte$ is equal to its smallest triangulated category
 containing $\mot(X)$ for a smooth scheme $X$ and stable under
 coproducts.
\end{rem}

\subsection{The stable $\AA^1$-derived category}

\begin{num}
As usual in motivic homotopy theory, we now describe 
 the $\PP^1$-stabilization of the category of 
 $\mathrm{MW}$-motivic complexes for the topology $t$ (again, $t=\nis, \et$).

Recall that the Tate twist in $\DMtex t$ is defined by one of the following
 equivalent formulas:
$$
\tilde R(1):=\mot(\PP^1_k)/\mot(\{\infty\})[-2]
 \simeq \mot(\AA^1_k)/\mot(\AA^1_k-\{0\})[-2]
 \simeq \mot(\GG)/\mot(\{1\})[-1].
$$
In the construction of the $\PP^1$-stable category as well as
 in the study of the homotopy $t$-structure, it is useful to
 introduce a redundant notation of $\GG$-twist:
$$
\tilde R\{1\}:=\mot(\GG)/\mot(\{1\})
$$
so that $\tilde R\{1\}=\tilde R(1)[1]$. The advantage of this
 definition is that we can consider $\tilde R\{1\}$
 as a $\mathrm{MW}$-$t$-sheaf instead of a complex. For $m\geq 1$, we set $\tilde R\{m\}=\tilde R\{1\}^{\otimes m}$ and we observe that $\tilde R(m)=\tilde R\{m\}[-m]$. 

Let us recall the general process of $\otimes$-inversion of the
 Tate twist in the context of model categories,
 as described in our particular case in \cite[\textsection 7]{CD1}.
We define the category $\spt$ of (abelian) Tate $\mathrm{MW}$-$t$-spectra as the additive
 category whose object are couples $(\mathcal F_*,\epsilon_*)$ where $\mathcal F_*$
 is a sequence of $\mathrm{MW}$-$t$-sheaves such that $\mathcal F_n$ is equipped with
 an action of the symmetric group $\mathfrak S_n$
 and, for any integer $n \geq 0$,
$$
\epsilon_n:(\mathcal F_n\{1\}:=\mathcal F_n \otimes \tilde R\{1\})
 \rightarrow \mathcal F_{n+1}
$$
is a morphism of $\mathrm{MW}$-$t$-sheaves, called the \emph{suspension map},
 such that the iterated morphism
$$
\mathcal F_n\{m\} \rightarrow \mathcal F_{n+m}
$$
is $\mathfrak S_n \times \mathfrak S_m$-equivariant for any $n\geq 0$ and $m\geq 1$.
 (see \emph{loc. cit.} for more details). The morphisms in $\spt$ between couples $(\mathcal F_*,\epsilon_*)$ and $(\mathcal G_*,\tau_*)$ are sequences of $\mathfrak S_n$-equivariant morphisms $f_n:\mathcal F_n\to \mathcal G_n$ such that the following diagram of $\mathfrak S_n \times \mathfrak S_m$-equivariant maps
 \[
 \xymatrix{\mathcal F_n\{m\} \ar[r]\ar[d]_-{f_n\{m\}} & \mathcal F_{n+m}\ar[d]^-{f_{n+m}} \\
 \mathcal G_n\{m\} \ar[r] & \mathcal G_{n+m}}
 \]
 is commutative for any $n\geq 0$ and $m\geq 1$.

This is a Grothendieck abelian, closed symmetric monoidal category with tensor product described in \cite[\S 7.3, \S 7.4]{CD1} (together with \cite[Chapter VII, \S 4, Exercise 6]{MacLane98}).
Further, we have a canonical adjunction of abelian categories:
\begin{equation}\label{eq:suspension}
\Sigma^\infty:\sht \leftrightarrows \spt:\Omega^\infty
\end{equation}
such that $\Sigma^\infty(\mathcal F)=(\mathcal F\{n\})_{n \geq 0}$
 with the obvious suspension maps
 and $\Omega^\infty(\mathcal F_*,\epsilon_*)=\mathcal F_0$.
 Recall the Tate $\mathrm{MW}$-$t$-spectrum $\Sigma^\infty(\mathcal F)$
 is called the \emph{infinite spectrum} associated with $\mathcal F$.
 The functor $\Sigma^\infty$ is monoidal (cf. \cite[\S 7.8]{CD1}).

One can define the $\AA^1$-stable cohomology of a complex
 $\mathbb E=(\mathbb E_*,\sigma_*)$ 
 of Tate $\mathrm{MW}$-$t$-spectra, for any smooth scheme $X$ and any
 couple $(n,m) \in \ZZ^2$:
\begin{equation}\label{eq:stable_A1_cohomology}
\H^{n,m}_{st-\AA^1}(X,\mathbb E)
 :=\ilim_{r \geq \max(0,-m)}
 \Big(\Hom_{\DMte}(\mot(X)\{r\},\mathbb E_{m+r}[n])\Big)
\end{equation}
where the transition maps are induced by the suspension maps $\sigma_*$ and $\mot(X)\{r\}=\mot(X)\otimes \tilde R\{r\}$. 
\end{num}

\begin{df}
We say that a morphism $\varphi:\mathbb E \rightarrow \mathbb F$ of complexes of Tate
 $\mathrm{MW}$-$t$-spectra is a \emph{stable $\AA^1$-equivalence}
 if for any smooth scheme $X$
 and any couple $(n,m) \in \ZZ^2$, the induced map
$$
\varphi_*:\H^{n,m}_{st-\AA^1}(X,\mathbb E)
 \rightarrow \H^{n,m}_{st-\AA^1}(X,\mathbb F)
$$
is an isomorphism.

One defines the category $\DMtx t$
 of \emph{$\mathrm{MW}$-motivic spectra} for the topology $t$
 as the localization of the triangulated category
 $\Der(\spt)$ with respect to stable $\AA^1$-derived equivalences.
\end{df}

\begin{num}\label{num:twist-1}
As usual, we can describe the above category as the homotopy
 category of a suitable model category.

First, recall that we can define the negative twist of an abelian Tate
 $\mathrm{MW}$-$t$-spectrum $\mathcal F_*$ by the formula, for $n>0$:
\begin{equation}\label{eq:twist-1}
\begin{split}
\lbrack(\mathcal F_*)\{-n\}\rbrack_m=
\begin{cases}
\ZZ[\mathfrak S_m] \otimes_{\ZZ[\mathfrak S_{m-n}]} \mathcal F_{m-n}
 & \text{if } m \geq n. \\
0 & \text{otherwise.}
\end{cases}
\end{split}
\end{equation}
with the obvious suspension maps. Note for future references that
 one has according to this definition and that of the tensor product:
\begin{equation}\label{eq:twist-2}
\mathcal F_*\{-n\} =\mathcal F_* \otimes (\Sigma^{\infty}\tilde R)\{-n\}.
\end{equation}

We then define the class of cofibrations in $\Comp(\spt)$
 as the smallest class of morphisms of complexes closed under
 suspensions, negative twists, pushouts,
 transfinite compositions and retracts generated by
 the infinite suspensions of cofibrations in $\Comp(\sht)$.

Applying \cite[Prop. 7.13]{CD1}, we get:
\end{num}
\begin{prop}
The category $\Comp(\spt)$ of complexes of Tate $\mathrm{MW}$-$t$-spectra
 has a symmetric monoidal model structure
 with stable $\AA^1$-equivalences as weak equivalences and
 cofibrations as defined above.
 This model structure is proper and cellular.

Moreover, the fibrations for this model structure
 are epimorphisms of complexes whose kernel
 is a complex $\mathbb E$ such that:
\begin{itemize}
\item for any $n \geq 0$, $\mathbb E_n$ is a $t$-flasque and $\AA^1$-local
 complex;
\item for any $n \geq 0$, the map obtained by adjunction from the
 obvious suspension map:
$$
\mathbb E_{n+1} \rightarrow \derR \uHom(\tilde R\{1\},\mathbb E_n)
$$
is an isomorphism.
\end{itemize}
\end{prop}
Therefore, the homotopy category $\DMtx t$ is a triangulated symmetric
 monoidal category with internal Hom. The adjoint pair \eqref{eq:suspension}
 can be derived and induces an adjunction of triangulated categories:
$$
\Sigma^\infty:\DMtex t \leftrightarrows \DMtx t:\Omega^\infty.
$$
As a left derived functor of a monoidal functor, the functor
 $\Sigma^\infty$ is monoidal.
Slightly abusing notation, we still denote by $\mot(X)$ the $\mathrm{MW}$-motivic spectrum $\Sigma^{\infty}(\mot(X))$.

\begin{num}
By construction, the $\mathrm{MW}$-motivic spectrum $\tilde R\{1\}$,
 and thus $\tilde R(1)$ is $\otimes$-invertible in $\DMtx t$
 (see \cite[Prop. 7.14]{CD1}). Moreover,
 using formulas \eqref{eq:twist-1} and \eqref{eq:twist-2},
 one obtains a canonical map  in $\spt$:
$$
\phi:\Sigma^\infty \tilde R\{1\} \otimes \big((\Sigma^\infty \tilde R)\{-1\}\big)
 \rightarrow \big(\Sigma^\infty \tilde R\{1\}\big)\{-1\}
 \rightarrow \tilde R.
$$
The following proposition justifies the definition of Paragraph
 \ref{num:twist-1} of negative twists.
\end{num}
\begin{prop}
The map $\phi$ is a stable $\AA^1$-equivalence.
 The $\mathrm{MW}$-motive $(\Sigma^\infty \tilde R)\{-1\}$
 is the tensor inverse of $\tilde R\{1\}$.
 For any $\mathrm{MW}$-$t$-spectra $\mathbb E$,
 the map obtained by adjunction from $\mathbb E \otimes \phi$:
$$
\mathbb E\{-1\} \rightarrow \derR \uHom (\tilde R\{1\},\mathbb E)
$$
is a stable $\AA^1$-equivalence.
\end{prop}
\begin{proof}
The first assertion follows from a direct computation using
 the definition of stable $\AA^1$-equivalences via the 
 $\AA^1$-stable cohomology \eqref{eq:stable_A1_cohomology}.
 The other assertions are formal consequences of the first one.
\end{proof}

As in the effective case, we derive from the functors of
 Corollary \ref{cor:adjunctions_corr}
 Quillen adjunctions for the stable $\AA^1$-local model
 structures and consequently adjunctions of triangulated categories:
\begin{equation}\label{eq:chg_top&tr_DM}
\begin{split}
\xymatrix@C=30pt@R=24pt{
\DA\ar@<+2pt>^{\derL \tilde \gamma^*}[r]\ar@<+2pt>^{a}[d]
 & \DMt\ar@<+2pt>^{\derL \pi^*}[r]\ar@<+2pt>^{\tilde a}[d]
     \ar@<+2pt>^{\tilde \gamma_{*}}[l]
 & \DM\ar@<+2pt>^{a^{tr}}[d]
     \ar@<+2pt>^{\pi_{*}}[l] \\
\DAx{\et}\ar@<+2pt>^{\derL \tilde \gamma^*_\et}[r]
    \ar@<+2pt>^{\derR \mathcal O}[u]
 & \DMtx{\et}\ar@<+2pt>^{\derL \pi^*_\et}[r]
	  \ar@<+2pt>^{\derR \mathcal O}[u]
		\ar@<+2pt>^{\tilde \gamma_{\et*}}[l]
 & \DMx{\et}
    \ar@<+2pt>^{\derR \mathcal O}[u]
		\ar@<+2pt>^{\pi_{\et*}}[l]
}
\end{split}
\end{equation}
where each left adjoints is monoidal and sends a motive of a smooth
 scheme to the corresponding motive.

Formally, one can compute morphisms  of $\mathrm{MW}$-motivic spectra
 as follows.
\begin{prop}
For any smooth scheme $X$, any pair of integers $(n,m) \in \ZZ^2$
 and any $\mathrm{MW}$-motivic spectrum $\mathbb E$, one has a canonical
 functorial isomorphism:
\begin{align*}
\Hom_{\DMt}(\mot(X),\mathbb E(n)[m])
 &\simeq \H^{n,m}_{st-\AA^1}(X,\mathbb E) \\
 &=\ilim_{r \geq \max(0,-m)}
 \Big(\Hom_{\DMte}(\mot(X)\{r\},\mathbb E_{m+r}[n])\Big)
\end{align*}
\end{prop}
\begin{proof}
 This follows from \cite[4.3.61 and 4.3.79]{Ayoub}
 which can be applied because of Remark \ref{rem:compact_gen_DMte}
 and the fact the cyclic permutation of order 3 acts on
 $\tilde R\{3\}$ as the identity in $\DMte$ (the proof of this fact is postponed until Corollary \ref{cor:permut_Tate}).
\end{proof}

In fact, as in the case of motivic complexes,
 one gets a better result if we assume the
 base field $k$ is infinite perfect.
 This is due to  the following
 analogue of Voevodsky's cancellation theorem \cite{Voevodsky10}, 
 which is proved in \cite{Fasel16}. 
\begin{thm}\label{thm:cancellation}
Let $k$ be a perfect infinite field. Then for any complexes
 $K$ and $L$ of $\mathrm{MW}$-sheaves, the morphism 
\[
\Hom_{\DMte}(K,L) \rightarrow \Hom_{\DMte}(K(1),L(1)),
\]
obtained by tensoring with the Tate twist, is an isomorphism.
\end{thm}

We then formally deduce the following corollary from this result.

\begin{cor}\label{cor:cancellation}
If $k$ is an infinite perfect field, the functor
$$
\Sigma^\infty:\DMte \rightarrow \DMt
$$
is fully faithful.
\end{cor}

\section{$\mathrm{MW}$-motivic cohomology}\label{sec:MWcohom}

\subsection{$\mathrm{MW}$-motivic cohomology as ext-groups}

Given our construction of the triangulated category $\DMt$,
 we can now define, in the style of Beilinson,
 a generalization of motivic cohomology as follows.
 
\begin{df}\label{def:generalizedMW}
We define the \emph{$\mathrm{MW}$-motivic cohomology} of a smooth scheme $X$
 in degree $(n,i) \in \ZZ^2$ and coefficients in $R$ as:
$$
\H_{\mathrm{MW}}^{n,i}(X,R)=\Hom_{\DMt}(\mot(X),\tilde R(i)[n]).
$$
\end{df}
As usual, we deduce a cup-product on $\mathrm{MW}$-motivic cohomology.
 We define its \'etale variant by taking morphisms in $\DMtx{\et}$.
 Then we derive from the preceding (essentially) commutative diagram the
 following morphisms of cohomology theories, all compatible with products
 and pullbacks:
\begin{equation}\label{eq:MW-regulators}
\xymatrix@=14pt{
\H_{\AA^1}^{n,i}(X,R)\ar[r]\ar[d] & \H_{\mathrm{MW}}^{n,i}(X,R)\ar[r]\ar[d]
 & \H_M^{n,i}(X,R)\ar[d] \\
\H_{\AA^1,\et}^{n,i}(X,R)\ar[r] & \H_{\mathrm{MW},\et}^{n,i}(X,R)\ar[r]
 & \H_L^{n,i}(X,R).
}
\end{equation}
where $\H_{\AA^1}(X,R)$ and $\H_{\AA^1,\et}(X,R)$ are respectively Morel's stable $\AA^1$-derived cohomology and its \'etale version
 while $\H_M^{n,i}(X,R)$ and $\H_L^{n,i}(X,R)$ are respectively 
 the motivic cohomology and the Lichtenbaum motivic cohomology
 (also called \emph{\'etale motivic cohomology}).

Gathering all the informations we have obtained in the previous
 section on $\mathrm{MW}$-motivic complexes, we get the following computation.
 
\begin{prop}\label{prop:explicit}
Assume that $k$ is an infinite perfect field.
For any smooth scheme $X$ and any couple of integers $(n,m) \in \ZZ^2$,
 the $\mathrm{MW}$-motivic cohomology $\H^{n,m}_{\mathrm{MW}}(X,\ZZ)$ defined previously coincides
 with the generalized motivic cohomology defined in \cite{Calmes14b}.

More explicitly,
$$
\H^{n,m}_{\mathrm{MW}}(X,\ZZ)=\begin{cases}
\H^n_\zar(X,\tilde\ZZ(m)) & \text{if } m>0, \\
\H^n_\zar(X,\sKMW_0) & \text{if } m=0, \\
\H^{n-m}_\zar(X,\Wi) & \text{if } m<0
\end{cases}
$$
where $\sKMW_0$ (resp. $\Wi$) is the unramified 
 sheaf associated with Milnor-Witt K-theory in degree $0$ 
 (resp. unramified Witt sheaf) -- see \cite[\textsection 3]{Morel08}.
\end{prop}

\begin{proof}
The cases $m>0$ and $m=0$ are clear from the previous corollary
 and Corollary \ref{cor:W-motivic&generalized}.

Consider the case $m<0$. Then we can use the following computation:
\begin{align*}
\H_{\mathrm{MW}}^{n,m}(X,\ZZ)&=\Hom_{\DMt}\big(\mot(X),\tilde \ZZ\{m\}[n-m]\big) \\
&=\Hom_{\DMt}\big(\mot(X)\otimes \tilde\ZZ\{-m\},\tilde \ZZ[n-m]\big) \\
&=\Hom_{\DMte}\big(\mot(X),\derR\uHom(\tilde \ZZ\{-m\},\tilde \ZZ)[n-m]\big)
\end{align*}
where the last identification follows from the preceding corollary and the usual adjunction.

As the $\mathrm{MW}$-motivic complex $\tilde \ZZ\{-m\}$ is cofibrant
 and the motivic complex $\tilde \ZZ=\sKMW_0$ is Nisnevich-local
 and $\AA^1$-invariant (cf. \cite[Ex. 4.4]{Calmes14b} and \cite[Cor. 11.3.3]{FaselCW}),
 we get:
$$
\derR\uHom(\tilde \ZZ\{-m\},\tilde \ZZ)=\uHom(\tilde \ZZ\{-m\},\tilde \ZZ)
$$
and this last sheaf is isomorphic to $\Wi$ according to
\cite[Lemma 5.23]{Calmes14b}. So the assertion now follows from
 Corollaries \ref{cor:A1-local_complexes} and \ref{cor:compare_Hom&cohomology}.\end{proof}

\begin{num}
We next prove a commutativity result for $\mathrm{MW}$-motivic cohomology. 
First, note that the sheaf 
$\tilde{R}\{1\}=\tilde{R}(\GGx k)/\tilde{R}(\{1\})$ is a direct factor of 
$\tilde{R}(\GGx k)$ and that the permutation map
\[
\sigma:\tilde{R}(\GGx k)\otimes \tilde{R}(\GGx k)\to \tilde{R}(\GGx k)\otimes 
\tilde{R}(\GGx k)
\]
given by the morphism of schemes $\GGx k\times \GGx k\to \GGx k\times \GGx k$ 
defined by $(x,y)\mapsto (y,x)$ induces a map
\[
\sigma:\tilde{R}\{1\}\otimes \tilde{R}\{1\}\to \tilde{R}\{1\}\otimes 
\tilde{R}\{1\}.
\] 
On the other hand, recall from Remark \ref{rem:finite_corr&plim} (5), that 
$\smc$ is $\sKMW_0(k)$-linear. In particular, the class of $\epsilon=-\langle 
-1\rangle\in \sKMW_0(k)$ (and its corresponding element in 
$\sKMW_0(k)\otimes_{\ZZ} R$ that we still denote by $\epsilon$) yields a 
$\mathrm{MW}$-correspondence 
\[
\epsilon=\epsilon\cdot Id:\tilde{R}(\GGx k)\otimes \tilde{R}(\GGx k)\to 
\tilde{R}(\GGx k)\otimes \tilde{R}(\GGx k)
\]
that also induces a $\mathrm{MW}$-correspondence
\[
\epsilon:\tilde{R}\{1\}\otimes \tilde{R}\{1\}\to \tilde{R}\{1\}\otimes 
\tilde{R}\{1\}.
\]
We can now state the following lemma (\cite[Lemma 3.0.6]{Fasel16}).
\end{num}
\begin{lm}
The $\mathrm{MW}$-correspondences $\sigma$ and $\epsilon$ are $\AA^1$-homotopic.
\end{lm} 

As an obvious corollary, we obtain the following result.

\begin{cor}\label{cor:permut_Tate}
For any $i,j \in \ZZ$, the switch 
$\tilde{R}(i)\otimes \tilde{R}(j)\to \tilde{R}(j)\otimes \tilde{R}(i)$ 
is $\AA^1$-homotopic to $\langle (-1)^{ij}\rangle$. 
\end{cor}

\begin{proof}
By definition, we have $\tilde{R}(i):=\tilde{R}\{i\}[-i]$ and 
$\tilde{R}(j):=\tilde{R}\{j\}[-j]$. We know from the previous lemma that the 
switch $\tilde{R}\{i\}\otimes \tilde{R}\{j\}\to \tilde{R}\{j\}\otimes 
\tilde{R}\{i\}$ is homotopic to $(\epsilon)^{ij}$. The result now follows from 
the compatibility isomorphisms for tensor triangulated categories (see e.g. 
\cite[Exercise 8A.2]{Mazza06}) and the fact that 
$(-1)^{ij}(\epsilon)^{ij}=\langle (-1)^{ij}\rangle$.
\end{proof}

\begin{thm}\label{thm:commutative}
Let $i,j\in \ZZ$ be integers. For any smooth scheme $X$, the pairing
\[
\mathrm{H}_{\mathrm{MW}}^{p,i}(X,R)\otimes 
\mathrm{H}_{\mathrm{MW}}^{q,j}(X,R)\to \mathrm{H}_{\mathrm{MW}}^{p+q,i+j}(X,R)
\]
is $(-1)^{pq}\langle (-1)^{ij}\rangle$-commutative.
\end{thm}

\begin{proof}
The proof of \cite[Theorem 15.9]{Mazza06} applies mutatis mutandis.
\end{proof}

\subsection{Comparison with Chow-Witt groups}

\subsubsection{The naive Milnor-Witt presheaf}.
Let $S$ be a ring and let $S^\times$ be the group of units in $S$. We define the naive Milnor-Witt presheaf of $S$ as in the case of fields by considering the free $\ZZ$-graded algebra $A(S)$ generated by the symbols $[a]$ with $a\in S^\times$ in degree $1$ and a symbol $\eta$ in degree $-1$ subject to the usual relations:

\begin{enumerate}
\item $[ab]=[a]+[b]+\eta [a][b]$ for any $a,b\in S^\times$.
\item $[a][1-a]=0$ for any $a\in S^\times\setminus \{1\}$.
\item $\eta [a]=[a]\eta$ for any $a\in S^\times$
\item $\eta (\eta [-1]+2)=0$.
\end{enumerate}

This definition is clearly functorial in $S$ and it follows that we obtain a presheaf of $\ZZ$-graded algebras on the category of smooth schemes via
\[
X\mapsto \KMW_*(\OO(X)).
\]
We denote by $\sKMW_{*,\mathrm{naive}}$ the associated Nisnevich sheaf of graded $\ZZ$-algebras and observe that this definition naturally extends to essentially smooth $k$-schemes. Our next aim is to show that this naive definition in fact coincides with the definition of the unramified Milnor-Witt $K$-theory sheaf given in \cite[\S 3]{Morel08} (see also \cite[\S 1]{Calmes14b}). Indeed, let $X$ be a smooth integral scheme. The ring homomorphism $\OO(X)\to k(X)$ induces a ring homomorphism $\KMW_*(\OO(X))\to \KMW_*(k(X))$ and it is straightforward to check that elements in the image are unramified, i.e. that the previous homomorphism induces a ring homomorphism $\KMW_*(\OO(X))\to \sKMW_*(X)$. By the universal property of the associated sheaf, we obtain a morphism of sheaves
\[
\sKMW_{*,\mathrm{naive}}\to \sKMW_*.
\]
If $X$ is an essentially smooth local $k$-scheme, it follows from \cite[Theorem 6.3]{Gille15} that the map $\sKMW_{*,\mathrm{naive}}(X)\to \sKMW_*(X)$ is an isomorphism, showing that the above morphism is indeed an isomorphism.

\begin{num}\textbf{A comparison map}.
Let now $X$ be a smooth connected scheme and let $a\in \OO(X)^\times$ be an invertible global section. It corresponds to a morphism $X\to \GGx k$ and in turn to an element in $\smc(X,\GGx k)$ yielding a map 
\[
s(a):\OO(X)^\times\to \mathrm{Hom}_{\DMteZ}(\tilde M(X),\tilde \ZZ\{1\})=\H^{1,1}_{\mathrm{MW}}(X,\ZZ).
\]
Consider next the correspondence $\eta[t]\in \wCH^0(X\times \GGx k)=\smc (X\times \GGx k,\spec k)$ and observe that it restricts trivially when composed with the map $X\to X\times \GGx k$ given by $x\mapsto (x,1)$. It follows that we obtain an element 
\[
s(\eta)\in \mathrm{Hom}_{\DMteZ}(\tilde M(X)\otimes \tilde \ZZ\{1\},\tilde \ZZ)=\mathrm{Hom}_{\DMt}(\tilde M(X),\tilde \ZZ\{-1\})=\H^{-1,-1}_{\mathrm{MW}}(X,\ZZ).
\]
Using the product structure of the cohomology ring, we finally obtain a (graded) ring homomorphism 
\[
s:A(\OO(X))\to \bigoplus_{n\in \ZZ}\H^{n,n}_{\mathrm{MW}}(X,\ZZ)
\]
that is easily seen to be functorial in $X$.  
\end{num}

\begin{thm}
Let $X$ be a smooth scheme. Then, the graded ring homomorphism
\[
s:A(\OO(X))\to \bigoplus_{n\in \ZZ}\H^{n,n}_{\mathrm{MW}}(X,\ZZ)
\]
induces a graded ring homomorphism
\[
s:\KMW_*(\OO(X))\to \bigoplus_{n\in \ZZ}\H^{n,n}_{\mathrm{MW}}(X,\ZZ)
\] 
which is functorial in $X$.
\end{thm}

\begin{proof}
We have to check that the four relations defining Milnor-Witt $K$-theory hold in the graded ring on the right-hand side. First, note that Theorem \ref{thm:commutative} yields $\epsilon s(\eta)s(a)=s(a)s(\eta)$ and the third relation follows from the fact that $\epsilon s(\eta)=s(\eta)$ by construction. Observe next that $s(\eta)s(-1)+1=\langle -1\rangle$ by \cite[Lemma 6.0.1]{Fasel16} and it follows easily that $s(\eta)(s(\eta)s(-1)+2)=0$. Next, consider the multiplication map
\[
m:\GGx k\times \GGx k\to \GGx k
\]
and the respective projections on the $j$-th factor
\[
p_j:\GGx k\times \GGx k\to \GGx k
\]
for $j=1,2$.
They all define correspondences that we still denote by the same symbols and it is straightforward to check that $m-p_1-p_2$ defines a morphism $\tilde \ZZ\{1\}\otimes \ZZ\{1\}\to \tilde\ZZ\{1\}$ in $\smc$. It follows from \cite[Lemma 6.0.2]{Fasel16} that this correspondence corresponds to $s(\eta)$ through the isomorphism
\[
\mathrm{Hom}_{\DMteZ}(\tilde\ZZ\{1\},\tilde\ZZ)\to \mathrm{Hom}_{\DMteZ}(\tilde \ZZ\{1\}\otimes \ZZ\{1\},\tilde\ZZ\{1\})
\]
given by the cancellation theorem. As a corollary, we see that the defining relation (1) of Milnor-Witt $K$-theory is satisfied in $\bigoplus_{n\in \ZZ}\H^{n,n}_{\mathrm{MW}}(X,\ZZ)$. Indeed, if $a,b\in \OO(X)^\times$, then $s(a)s(b)$ is represented by the morphism $X\to \GGx k\times \GGx k$ corresponding to $(a,b)$. Applying $m-p_1-p_2$ to this correspondence, we get $s(ab)-s(a)-s(b)$ which is $s(\eta)s(a)s(b)$ by the above discussion.

To check that the Steinberg relation holds in the right-hand side, we first consider the morphism
\[
\AA^1\setminus \{0,1\}\to \AA^1\setminus \{0\}\times \AA^1\setminus \{0\}
\]
defined by $a\mapsto (a,1-a)$. Composing with the correspondence $\tilde{M}( \AA^1\setminus \{0\}\times \AA^1\setminus \{0\})\to \tilde{M}((\GGx k)^{\wedge 2})$, we obtain a morphism
\[
\tilde{M}(\AA^1\setminus \{0,1\})\to \tilde{M}((\GGx k)^{\wedge 2}).
\]
We can perform the same computation in $\DAeZ$ where this morphism is trivial by \cite{Hu2001} and we conclude that it is also trivial in $\DMteZ$ by applying the functor $\derL\tilde\gamma^*$.
\end{proof}

For any $p,q\in \ZZ$, we denote by $\mathbf{H}_{\mathrm{MW}}^{p,q}$ the (Nisnevich) sheaf associated to the presheaf $X\mapsto \H_{\mathrm{MW}}^{p,q}(X,\ZZ)$. The homomorphism of the previous theorem induces a morphism on induced sheaves and we have the following result.

\begin{thm}\label{thm:KMWmotivic}
The homomorphism of sheaves of graded rings
\[
s:\sKMW_*\to \bigoplus_{n\in \ZZ}\mathbf{H}_{\mathrm{MW}}^{n,n}
\]
is an isomorphism.
\end{thm}

\begin{proof}
Let $L/k$ be a finitely generated field extension. Then, it follows from \cite[Theorem 6.19]{Calmes14b} that the homomorphism $s_L$ is an isomorphism. Now, the presheaf on $\smc$ given by $X\mapsto \bigoplus_{n\in \ZZ}\H^{n,n}_W(X,\ZZ)$ is homotopy invariant by definition. It follows from Theorem \ref{thm:A1local_framedPSh} that the associated sheaf is strictly $\AA^1$-invariant. Now, $\sKMW_*$ is also strictly $\AA^1$-invariant and it follows from \cite[Definition 2.1, Remark 2.3, Theorem 2.11]{Morel08} that $s$ is an isomorphism if and only if $s_L$ is an isomorphism for any finitely generated field extension $L/k$.
\end{proof}


\begin{thm}\label{thm:hyper}
For any smooth scheme $X$ and any integers $p,n\in \ZZ$, the hypercohomology spectral sequence induces isomorphisms
\[
\H^{p,n}_{\mathrm{MW}}(X,\ZZ)\to \H^{p-n}(X,\sKMW_n)
\]
provided $p\geq 2n-1$. 
\end{thm}

\begin{proof}
In view of Proposition \ref{prop:explicit}, we may suppose that $n>0$. For any $q\in\ZZ$, we denote by $\mathbf{H}_{\mathrm{MW}}^{q,n}$ the Nisnevich sheaf associated to the presheaf $X\mapsto \H^{q,n}_{\mathrm{MW}}(X,\ZZ)$ and observe that they coincide with the cohomology sheaves of the complexes $\tilde\ZZ(n)$. Now, the latter are concentrated in cohomological levels $\leq n$ and it follows that $\mathbf{H}_{\mathrm{MW}}^{q,n}=0$ if $q>n$. On the other hand, the sheaves $\mathbf{H}_{\mathrm{MW}}^{q,n}$ are strictly $\AA^1$-invariant, and as such admit a Gersten complex whose components in degree $m$ are of the form 
\[
\bigoplus_{x\in X^{(p)}}(\mathbf{H}_{\mathrm{MW}}^{q,n})_{-p}(k(x),\wedge^p (\mathfrak m_x/\mathfrak m_x^2)^*)
\] 
by \cite[\S 5]{Morel08}. By the cancellation theorem \ref{thm:cancellation}, we have a canonical isomorphism of sheaves $\mathbf{H}_{\mathrm{MW}}^{q-p,n-p}\simeq (\mathbf{H}_{\mathrm{MW}}^{q,n})_{-p}$ and it follows that the terms in the Gersten resolution are of the form 
\[
\bigoplus_{x\in X^{(p)}}(\mathbf{H}_{\mathrm{MW}}^{q-p,n-p})(k(x),\wedge^p (\mathfrak m_x/\mathfrak m_x^2)^*).
\] 
If $p\geq n$, then $\tilde \ZZ(n-p)\simeq \sKMW_{n-p}[p-n]$ and it follows that $\mathbf{H}_{\mathrm{MW}}^{q-p,n-p}$ is the sheaf associated to the presheaf $X\mapsto H^{q-n}(X,\sKMW_{n-p})$, which is trivial if $q\neq n$. Altogether, we see that 
\[
\mathbf{H}_{\mathrm{MW}}^{q,n}=\begin{cases} 0 & \text{if } q>n. \\
0 & \text{if } p\geq n \text{ and } q\neq n. \end{cases} 
\]
We now consider the hypercohomology spectral sequence for the complex $\tilde\ZZ(n)$ (\cite[Theorem 0.3]{Suslin00}) $E_2^{p,q}:=\H^p(X,\mathbf{H}_{\mathrm{MW}}^{q,n})\implies \H^{p+q,n}(X,\ZZ)$ which is strongly convergent. Our computations of the sheaves $\mathbf{H}_{\mathrm{MW}}^{q,n}$ immediately imply that $\H^{p-n}(X,\mathbf{H}_{\mathrm{MW}}^{n,n})=\H^{p,n}_{\mathrm{MW}}(X,\ZZ)$ if $p\geq 2n-1$. We conclude using Theorem \ref{thm:KMWmotivic}.
\end{proof}

\begin{rem}
The isomorphisms $\H^{p,n}_{\mathrm{MW}}(X,\ZZ)\to \H^{p-n}(X,\sKMW_n)$ are functorial in $X$. Indeed, the result comes from the analysis of the hypercohomology spectral sequence for the complexes $\tilde\ZZ(n)$, which is functorial in $X$.
\end{rem}

Setting $p=2n$ in the previous theorem, and using the fact that $\H^{n}(X,\sKMW_n)=\wCH^n(X)$ by definition (for $n\in \NN$), we get the following corollary.

\begin{cor}\label{cor:ChowWitt}
For any smooth scheme $X$ and any $n\in \NN$, the hypercohomology spectral sequence induces isomorphisms 
\[
\H^{2n,n}_{\mathrm{MW}}(X,\ZZ)\to \wCH^n(X).
\]
\end{cor}

\begin{rem}\label{rem:preThom}
Both Theorems \ref{thm:hyper} and Corollary \ref{cor:ChowWitt} are still valid if one considers cohomology with support on a closed subset $Y\subset X$, i.e. the hypercohomology spectral sequence (taken with support) yields an isomorphism
\[
\H^{p,n}_{\mathrm{MW},Y}(X,\ZZ)\to \H^{p-n}_Y(X,\sKMW_n)
\]
provided $p\geq 2n-1$.

Let now $E$ be a rank $r$ vector bundle over $X$, $s:X\to E$ be the zero section and $E^0=E\setminus s(X)$. The Thom space of $E$ is the object of $\DMteZZ$ defined by 
\[
Th(E)=\Sigma^\infty\mot(\tilde \ZZ(E)/\tilde \ZZ(E^0)).
\] 
It follows from Corollary \ref{cor:compare_Hom&cohomology} and \cite[Proposition 3.13]{Asok13} that (for $n\in\NN$)
\[
\Hom_{\DMteZZ}(\Th(E),\tilde \ZZ(n)[2n])\simeq \Hom_{\DMte}(\mot(\tilde \ZZ(E)/\tilde \ZZ(E^0)),\tilde \ZZ(n)[2n])\simeq \H_{\mathrm{MW},X}^{2n,n}(E,\ZZ).
\]
Using the above result, we get $\Hom_{\DMteZZ}(\Th(E),\tilde \ZZ(n)[2n])\simeq  \wCH^{n}_X(E)$. Using finally the Thom isomorphism (\cite[Corollary 5.30]{Morel08} or \cite[Remarque 10.4.8]{FaselCW})
\[
\wCH^{n-r}(X,\mathrm{det}(E))\simeq \wCH^{n}_X(L),
\]
we obtain an isomorphism
\[
\Hom_{\DMteZZ}(\Th(E),\tilde \ZZ(n)[2n])\simeq \wCH^{n-r}(X,\mathrm{det}(E))
\]
which is functorial (for schemes over $X$).
\end{rem}

\begin{rem}
The isomorphisms of Corollary \ref{cor:ChowWitt} induce a ring homomorphism
\[
\bigoplus_{n\in \NN}\H^{2n,n}_{\mathrm{MW}}(X,\ZZ)\to \bigoplus_{n\in\NN}\wCH^n(X).
\]
This follows readily from the fact that the isomorphism of Theorem \ref{thm:KMWmotivic} is an isomorphism of graded rings.
\end{rem}

\bibliographystyle{amsalpha}
\bibliography{DMt}

\providecommand{\bysame}{\leavevmode\hbox to3em{\hrulefill}\thinspace}
\providecommand{\MR}{\relax\ifhmode\unskip\space\fi MR }
\providecommand{\MRhref}[2]{%
  \href{http://www.ams.org/mathscinet-getitem?mr=#1}{#2}
}
\providecommand{\href}[2]{#2}
\begin{thebibliography}{MVW06}

\bibitem[AF16]{Asok13}
A.~Asok and J.~Fasel, \emph{Comparing {E}uler classes}, Quart. J. Math.
  \textbf{67} (2016), 603--635.

\bibitem[Ayo07]{Ayoub}
J.~Ayoub, \emph{Les six op\'erations de {G}rothendieck et le formalisme des
  cycles \'evanescents dans le monde motivique ({II})}, Ast\'erisque, vol. 315,
  Soc. Math. France, 2007.

\bibitem[CD09a]{CD1}
D.-C. Cisinski and F.~D{\'e}glise, \emph{Local and stable homological algebra
  in {G}rothendieck abelian categories}, HHA \textbf{11} (2009), no.~1,
  219--260.

\bibitem[CD09b]{CD3}
\bysame, \emph{Triangulated categories of mixed motives}, arXiv:0912.2110,
  version 3, 2009.

\bibitem[CF14]{Calmes14b}
B.~Calm\`es and J.~Fasel, \emph{The category of finite {C}how-{W}itt
  correspondences}, Available at \texttt{http://arxiv.org/abs/1412.2989}, 2014.

\bibitem[DHI04]{DHI}
D.~Dugger, S.~Hollander, and D.~C. Isaksen, \emph{Hypercovers and simplicial
  presheaves}, Math. Proc. Cambridge Philos. Soc. \textbf{136} (2004), no.~1,
  9--51.

\bibitem[Fas08]{FaselCW}
J.~Fasel, \emph{Groupes de {C}how-{W}itt}, M\'em. Soc. Math. Fr. (N.S.) (2008),
  no.~113, viii+197.

\bibitem[F{\O}16]{Fasel16}
J.~Fasel and P.~A. {\O}stv{\ae}r, \emph{A cancellation theorem for
  {$W$}-motives}, In preparation, 2016.

\bibitem[GP14]{Garkusha14}
G.~Garkusha and I.~Panin, \emph{Framed {M}otives of algebraic varieties},
  arXiv:1409.4372, 2014.

\bibitem[GP15]{Garkusha15}
\bysame, \emph{Homotopy invariant presheaves with framed transfers}, Preprint
  available at \texttt{http://arxiv.org/abs/1504.00884}, 2015.

\bibitem[GSZ15]{Gille15}
S.~Gille, S.~Scully, and C.~Zhong, \emph{Milnor-{W}itt groups over local
  rings}, Appeared in Advances, 2015.

\bibitem[Hir03]{Hirschhorn03}
P.~S. Hirschhorn, \emph{{Model categories and their localizations}},
  Providence, RI: American Mathematical Society (AMS), 2003.

\bibitem[HK01]{Hu2001}
P.~Hu and I.~Kriz, \emph{The {S}teinberg relation in {$\mathbb A^1$}-stable
  homotopy}, Int. {M}ath. {R}es. {N}ot. \textbf{17} (2001), 907--912.

\bibitem[Kol17]{Kolderup17}
H.~A. Kolderup, \emph{Homotopy invariance of presheaves with {M}ilnor-{W}itt
  transfers}, In preparation, 2017.

\bibitem[Mil12]{Milne12}
J.~S. Milne, \emph{Lectures on {E}tale {C}ohomology}, Available at
  \texttt{www.jmilne.org/math/}, 2012.

\bibitem[ML98]{MacLane98}
S.~Mac~Lane, \emph{{Categories for the working mathematician. 2nd ed.}}, 2nd ed
  ed., Graduate {T}exts in {M}athematics, New York, NY: Springer, 1998.

\bibitem[Mor05]{Morel05b}
F.~Morel, \emph{The stable {${\mathbb A}^1$}-connectivity theorems}, $K$-Theory
  \textbf{35} (2005), no.~1-2, 1--68. \MR{2240215 (2007d:14041)}

\bibitem[Mor12]{Morel08}
\bysame, \emph{$\mathbb {A}^1$-{A}lgebraic {T}opology over a {F}ield}, Lecture
  Notes in Math., vol. 2052, Springer, New York, 2012.

\bibitem[MVW06]{Mazza06}
C.~Mazza, V.~Voevodsky, and C.~Weibel, \emph{Lecture notes on motivic
  cohomology}, Clay Mathematics Monographs, vol.~2, American Mathematical
  Society, Providence, RI, 2006. \MR{MR2242284 (2007e:14035)}

\bibitem[Nee01a]{Nee2}
A.~Neeman, \emph{On the derived category of sheaves on a manifold}, Doc. Math.
  \textbf{6} (2001), 483--488.

\bibitem[Nee01b]{Nee}
A.~Neeman, \emph{Triangulated categories}, Annals of Mathematics Studies, vol.
  148, Princeton University Press, Princeton, NJ, 2001.

\bibitem[SV00]{Suslin00}
Andrei Suslin and Vladimir Voevodsky, \emph{Bloch-{K}ato {C}onjecture and
  {M}otivic {C}ohomology with {F}inite {C}oefficients}, The {A}rithmetic and
  {G}eometry of {A}lgebraic {C}ycles (B.~Brent Gordon, James~D. Lewis, Stefan
  M{\"u}ller-Stach, Shuji Saito, and Noriko Yui, eds.), Springer Netherlands,
  Dordrecht, 2000, pp.~117--189.

\bibitem[Voe10]{Voevodsky10}
V.~Voevodsky, \emph{Cancellation theorem}, Doc. Math. Extra Volume: Andrei A.
  Suslin Sixtieth Birthday (2010), 671--685.

\bibitem[VSF00]{FSV}
V.~Voevodsky, A.~Suslin, and E.~M. Friedlander, \emph{Cycles, transfers and
  motivic homology theories}, Annals of Mathematics Studies, vol. 143,
  Princeton Univ. Press, 2000.

\bibitem[Wei94]{Weibel94}
C.~Weibel, \emph{An introduction to homological algebra}, Cambridge University
  Press, 1994.

\end{thebibliography}

\end{document}